\documentclass[11pt,reqno]{amsproc}
\allowdisplaybreaks
\numberwithin{equation}{section}

\usepackage{empheq}
\usepackage{amssymb}
\usepackage{appendix} 

\usepackage{enumerate}
\usepackage{fullpage}
\usepackage{textcomp}
\usepackage{lettrine}
\usepackage{float} 
\usepackage{graphicx}
\usepackage{subfigure}
\usepackage{morefloats}

\usepackage[font=small,labelfont=bf,
   justification=justified,
   format=plain]{caption} % 'format=plain' avoids hanging indentation

\usepackage[debug=false, colorlinks=true, pdfstartview=FitV,
linkcolor=blue, citecolor=blue, urlcolor=blue]{hyperref}
\usepackage[semicolon,square,authoryear,sort]{natbib}

\usepackage[doipre={DOI:~}]{uri}

\usepackage{color}
\usepackage{colortbl}
\usepackage[most]{tcolorbox}

\newtcbox{\mymath}[1][]{%
    nobeforeafter, math upper, tcbox raise base,
    enhanced, colframe=blue!30!black,
    colback=blue!30, boxrule=1pt,
    #1}

\newlength{\drop}
\definecolor{amethyst}{rgb}{0.6, 0.4, 0.8}
\definecolor{burgundy}{rgb}{0.5, 0.0, 0.13}

\usepackage{longtable} 
\usepackage{multirow} 
\usepackage{rotating} 
\usepackage{bigstrut} 
\usepackage{hhline} 

\usepackage{amsthm}

\newtheoremstyle{remboldstyle}
  {}{}{}{}{\bfseries}{.}{.5em}{{\thmname{#1 }}{\thmnumber{#2}}{\thmnote{ (#3)}}}
\theoremstyle{remboldstyle}

\newtheorem{theorem}{Theorem}
\newtheorem{lemma}[theorem]{Lemma}

\newtheorem{remark}{Remark}

\title{\textbf{Flow Through Porous Media: A Hopf-Cole Transformation Approach for Modeling Pressure-Dependent Viscosity}}

\author{\textbf{Venkat S.~Maduri} and \textbf{Kalyana B.~Nakshatrala} \\
  {\small Department of Civil and Environmental Engineering \\
  University of Houston, Houston, Texas 77204, USA.}\\
  {\small \textbf{Correspondence to}: knakshatrala@uh.edu, 
  +1-713-743-4418}}


\begin{document}

%===========================;
%  Title page of the paper  ;
%===========================;
\begin{titlepage}
  \drop=0.1\textheight
  \centering
  \vspace*{\baselineskip}
  \rule{\textwidth}{1.6pt}\vspace*{-\baselineskip}\vspace*{2pt}
  \rule{\textwidth}{0.4pt}\\[\baselineskip]
       {\Large \textbf{\color{burgundy}
       Flow Through Porous Media: A Hopf-Cole Transformation 
       \\[0.3\baselineskip] Approach for Modeling 
       Pressure-Dependent Viscosity}}\\[0.3\baselineskip]
       \rule{\textwidth}{0.4pt}\vspace*{-\baselineskip}\vspace{3.2pt}
       \rule{\textwidth}{1.6pt}\\[\baselineskip]
       \scshape
       An e-print of this paper is available on arXiv. \par
       \vspace*{1\baselineskip}
       Authored by \\[0.5\baselineskip]

       {\Large V.~S.~Maduri\par}
       {\itshape Graduate Student, Department of Civil \& Environmental Engineering \\
       University of Houston, Houston, Texas 77204.}\\[0.5\baselineskip]

  {\Large K.~B.~Nakshatrala\par}
  {\itshape Department of Civil \& Environmental Engineering \\
  University of Houston, Houston, Texas 77204. \\
  \textbf{phone:} +1-713-743-4418, \textbf{e-mail:} knakshatrala@uh.edu \\
  \textbf{website:} http://www.cive.uh.edu/faculty/nakshatrala}\\[0.25\baselineskip]

  %----------------------;
  %  Graphical abstract  ;
  %----------------------;
  \begin{figure*}[ht]
    \centering
    \includegraphics[width=0.6\linewidth]{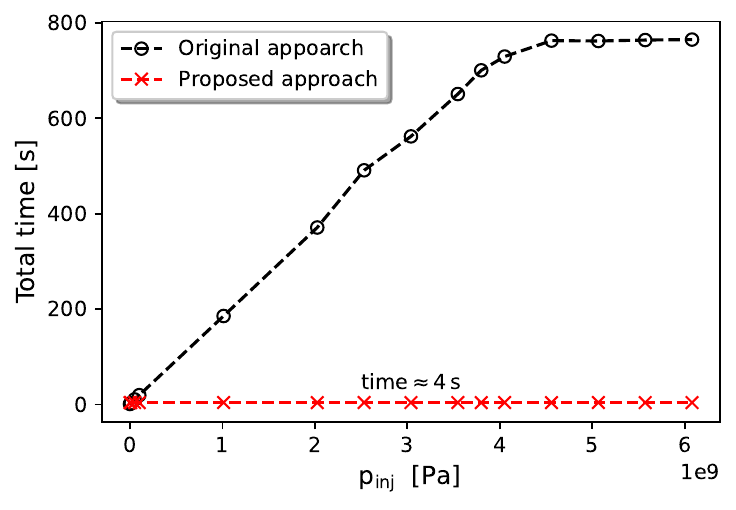}
    \captionsetup{labelformat=empty}
    \caption{This figure compares the time-to-solution between the original approach—solving the Barus model directly—and the proposed approach, which leverages the Hopf-Cole transformation. The proposed method significantly reduces the simulation time, requiring only a fraction of that used by the original approach. In addition to its computational advantages, the proposed method also offers several theoretical benefits, as illustrated in this paper.}
  \end{figure*}
  %%%
  \vfill
  {\scshape 2025} \\
  {\small Computational \& Applied Mechanics Laboratory} \par
\end{titlepage}

%=========================;
%  Abstract of the paper  ;
%=========================;
\begin{abstract} 
    Most organic liquids exhibit a pressure-dependent viscosity, making it crucial to consider this behavior in applications where pressures significantly exceed ambient conditions (e.g., geological carbon sequestration). Mathematical models describing flow through porous media while accounting for viscosity-pressure dependence are nonlinear (e.g., the Barus model). This nonlinearity complicates mathematical analysis and makes numerical solutions more time-intensive and prone to convergence issues. In this paper, we demonstrate that the Hopf-Cole transformation, originally developed for Burgers' equation, can recast the governing equations---describing flow through porous media with pressure-dependent viscosity---into a linear form. The transformed equations, resembling Darcy's equations in the transformed variables, enable (a) systematic mathematical analysis to establish uniqueness and maximum principles, (b) the derivation of a mechanics-based principle, and (c) the development of efficient numerical solutions using solvers optimized for Darcy equations. Notably, many properties of the linear Darcy equations naturally extend to nonlinear models that depend on pressure. For example, solutions to these nonlinear models adhere to a reciprocal relation analogous to that observed in Darcy's equations.
\end{abstract}

\maketitle

    %==================================;
    %  Include all the sections below  ;
    %==================================;
    \setcounter{figure}{0}   

    %*********************************************;
%                                             ;
%  NAME                                       ;
%    S1_Hopf_Cole_Intro.tex                   ;
%                                             ;
%*********************************************;
\section{INTRODUCTION AND MOTIVATION}
\label{Sec:S1_Hopf_Cole_Intro}

\lettrine[findent=2pt]{\fbox{\textbf{F}}}{l}uid flow through porous media manifests in a wide variety of engineering and medical applications critical to the economy, energy, environment, and health. These endeavors range from hydrocarbon extraction, groundwater management, water filtration, and $\mathrm{CO}_2$ sequestration, as well as tissue perfusion and the design of surgical masks, among others. However, given the breadth and intricate nature of these applications, a nuanced understanding of the underlying mechanisms is essential, with modeling playing a critical role.

One issue pertinent to this paper is that applications such as enhanced oil recovery and geological carbon sequestration often operate under high-pressure conditions---sometimes exceeding ambient pressure by two orders of magnitude. Additionally, the viscosity of most organic liquids (e.g., crude oil) exhibits a strong dependence on fluid pressure. Notably, \textsc{Carl Barus} demonstrated that this dependence follows an exponential relationship \citep{barus1893isothermals}. Consequently, accurately accounting for this pronounced functional dependence is imperative for high-pressure applications.

Since its introduction, the Darcy model---named after \textsc{Henry Darcy}---has served as the cornerstone for modeling fluid flow through porous media by relating fluid flux to the pressure gradient \citep{darcy1856public,simmons2008henry}. Due to its importance and mathematical elegance, Darcy equations (comprising of the balance of linear momentum, balance of mass, fluid's incompressibility, and the Darcy model---a constitutive relation) have garnered significant attention from mechanicians, applied mathematicians, and computational scientists. 
%------------------------------------------------;
%  Itemize: Current results for the Darcy model  ;
%------------------------------------------------;
\begin{itemize}
    \item Several results based on mechanics that the solutions of the Darcy equations satisfy are available; for example, the minimum dissipation theorem and the minimum mechanical power theorem \citep{shabouei2016mechanics}. 
    \item Noting that Darcy equations, when written in terms of pressure, give rise to second-order diffusion equations, numerous analysis studies have proven powerful theorems (such as existence, uniqueness, maximum principles, and comparison principles) on the nature of solutions \citep{evanspartial,gilbarg1977elliptic}. 
    \item Robust numerical solvers are available for the Darcy equations; for example, mixed formulations based on Raviart-Thomas spaces \citep{raviart1977primal,ern2004theory}, and stabilized mixed finite element formulations \citep{masud2002stabilized,nakshatrala2006stabilized}. 
\end{itemize}

However, the classical Darcy model assumes constant fluid properties. This assumption limits its applicability in real-world scenarios where fluid viscosity exhibits significant pressure dependence, such as in deep geothermal reservoirs, hydraulic fracturing, and enhanced oil recovery \citep{rajagopal2007hierarchy}. Consequently, various models incorporating pressure-dependent viscosity have been developed; for example, \citet{rajagopal2007hierarchy} introduced a hierarchy of mathematical models for studying flow through porous media, allowing fluid properties to depend on the solution fields (i.e., pressure and velocity). One popular model is the Barus model, which integrates the exponential relationship into the study of flow through porous media \citep{nakshatrala2011numerical}. This modification to the Darcy model introduces non-linearity into the system, making both analytical and numerical solutions more challenging but ultimately more representative of actual fluid behavior in the aforementioned applications. Expanding on this:
\begin{itemize}
    \item Conventional analytical approaches---such as Green's functions, Laplace transforms, and Fourier transforms---are not suitable for tackling these kinds of nonlinear problems. Thus, advanced alternative methods are warranted.
    \item For numerical solutions, iterative strategies are required to manage nonlinearity and ensure convergence. However, solving these nonlinear equations remains computationally onerous. Nonetheless, notable works have rigorously addressed the complexities of pressure-dependent viscosity through systematic numerical investigations and advanced stabilized formulations \citep{nakshatrala2011numerical,mapakshi2018scalable}.
    \item As mentioned earlier, the solution fields of the classical Darcy equations exhibit important mathematical properties and mechanics-based principles. However, equivalent properties for nonlinear models have yet to be fully established. In such cases, sophisticated mathematical techniques (e.g., nonlinear functional analysis) must be employed to achieve comparable results.
\end{itemize} 

The above discussion raises two central questions: 
\begin{enumerate} 
    \item[{[Q1]}] Is it possible to obtain theoretical results without employing new mathematical tools or to get numerical results without developing new formulations? 
    \item[{[Q2]}] Can we find a transformation that simplifies the governing equations arising from pressure-dependent models for fluid flow through porous media? 
\end{enumerate}
The main aim of this paper is to show that the answers to both of these questions are affirmative. Thus, the task at hand is to identify a suitable transformation. 

A noteworthy historical precedent for transform methods applicable for nonlinear problems can be traced back to \textsc{Gustav Kirchhoff}’s pioneering work on heat conduction, where he introduced a method to handle temperature-dependent material properties \citep{Kirchhoff1857,kirchhoff1894theorie}. The transformation converted a quasilinear diffusion equation into a linear diffusion equation when lower-order derivative terms are absent from the partial differential equation (i.e., there is no advection or reaction term). If such terms are present, the transformation yields a semi-linear diffusion equation instead. This transformation is now popularly referred to as the Kirchhoff transformation, and its presentation can be found in standard texts on diffusion equations \citep{crank1979mathematics}, heat transfer \citep{ozicsik1993heat} and applied mathematics \citep{kevorkian1990partial}. 

Nearly a century after Kirchhoff's work, \textsc{Eberhard Hopf} introduced a transformative method to study Burgers' equation \citep{Hopf1950}. The following year, \textsc{Julian Cole} independently discovered the same key idea while investigating aerodynamic flow problems \citep{Cole1951}. Their shared insight, now collectively referred to as the Hopf-Cole transformation, converts the nonlinear Burgers' equation into the linear heat equation, thus providing a powerful analytical framework for understanding and solving a wide array of nonlinear diffusion and shock phenomena. Interestingly, \textsc{Peter Vadasz} has recently demonstrated that the Kirchhoff transformation arises as a special case of the Hopf-Cole transformation, thereby unveiling a deeper connection between these two transform techniques \citep{Vadasz2010}.

%  Our approach
Influenced by such transform methods, our approach utilizes the Hopf-Cole transformation to develop a framework for studying fluid flow through porous media, where the fluid's viscosity depends on pressure. The transformation recasts the governing equations under the Barus model as classical Darcy equations---expressed in terms of the transformed variable. The proposed framework has both theoretical and practical significance. On the theoretical front, the simplicity of the equations obtained after applying the transformation allows us to derive various mathematical and mechanics-based properties satisfied by the solution fields of the original nonlinear model. On the practical front, the framework provides an efficient method for obtaining numerical solutions using linear solvers, thereby circumventing convergence issues with nonlinear iterative approaches.  

%  An outline of the paper
The organization of the rest of this paper is structured as follows. We present the governing equations of a nonlinear model describing fluid flow through porous media  (\S\ref{Sec:S2_Hopf_Cole_GE}). Next, we introduce a framework based on the Hopf-Cole transformation corresponding to the nonlinear mathematical model (\S\ref{Sec:S3_Hopf_Cole_Transformation}). We then establish key mathematical properties, including the maximum principle, and outline theoretical findings achieved by leveraging the simplicity of the transformed equations (\S\ref{Sec:S4_Hopf_Cole_Mathematical}). This is followed by mechanics-based principles, akin to those previously known for the classical Darcy equations (\S\ref{Sec:S5_Hopf_Cole_Mechanics}). Subsequently, we highlight the computational advantages of using the proposed framework (\S\ref{Sec:S6_Hopf_Cole_NR}). Finally, we conclude by summarizing the utility of the Hopf-Cole transformation in studying pressure-dependent flow through porous media models (\S\ref{Sec:S7_Hopf_Cole_Closure}).                

    %*********************************************;
%                                             ;
%  NAME                                       ;
%    S2_Hopf_Cole_GE.tex                      ;
%                                             ;
%*********************************************;
\section{GOVERNING EQUATIONS: PRESSURE-DEPENDENT VISCOSITY}
\label{Sec:S2_Hopf_Cole_GE}

Consider the flow of an incompressible fluid through a porous domain. We denote the domain by $\Omega \subset \mathbb{R}^{nd}$, where ``$nd$" represents the number of spatial dimensions. $\partial \Omega := \overline{\Omega} \setminus \Omega$ depicts the boundary, with $\overline{\Omega}$ symbolizing the set closure of $\Omega$, and $\mathbf{x} \in \overline{\Omega}$ refers to a spatial point. For differential operators, $\mathrm{grad}[\cdot]$ and $\mathrm{div}[\cdot]$ represent the spatial gradient and divergence operators, respectively. On the boundary, $\widehat{\mathbf{n}}(\mathbf{x})$ is the unit outward normal vector. 

In our description, we do not resolve the flow at the pore scale. Instead, we model the phenomenon at the macroscale, also known as the Darcy scale or plug flow scale \citep{bear2013dynamics}. At a spatial point, we indicate the Darcy velocity vector field by $\mathbf{v}(\mathbf{x})$ and the pressure field by $p(\mathbf{x})$. To describe homogenized flow, we use the permeability field, symbolized by $\mathbf{K}(\mathbf{x})$. We assume that this field is generally anisotropic; hence, it is a tensor field. Moreover, from a physical standpoint, the permeability field must be symmetric and positive definite and therefore invertible. However, from a mathematical perspective, we impose stronger restrictions on the permeability field: it must be uniformly elliptic (i.e., uniformly bounded below) and bounded above. Specifically, there exist two constants $k_1$ and $k_2$ with $0\, <\, k_{1} \, \leq \, k_2$ such that 
%--------------------------------;
%  Equation: Permeability field  ;
%--------------------------------;
\begin{align}
    k_{1} \,\boldsymbol{\zeta}(\mathbf{x}) \bullet \boldsymbol{\zeta}(\mathbf{x}) \leq \boldsymbol{\zeta}(\mathbf{x}) \bullet \mathbf{K}(\mathbf{x}) \, \boldsymbol{\zeta}(\mathbf{x}) \leq k_{2} \, \boldsymbol{\zeta}(\mathbf{x}) \bullet \boldsymbol{\zeta}(\mathbf{x}) \quad \forall \mathbf{x} \in \Omega, \forall \boldsymbol{\zeta}(\mathbf{x}) \in \mathbb{R}^{nd} \setminus \{\boldsymbol{0}\}
\end{align} 

We designate the density of the fluid by $\rho$ and the viscosity by $\mu$. As emphasized earlier, viscosity is not assumed to be constant but rather depends on the fluid pressure, a characteristic of organic liquids \citep{bridgman1926effect}. We use the Barus model to describe the dependence of the viscosity on the pressure \citep{barus1893isothermals}. Mathematically, the Barus model takes the following form: 
%-------------------------;
%  Equation: Barus model  ;
%-------------------------;
\begin{align}
    \label{Eqn:Hopf_Cole_Barus_model}
    \mu = \widehat{\mu}\big(p(\mathbf{x})\big) 
    := \mu_{0} \, \exp\left[\beta \, \left(\frac{p(\mathbf{x})}{p_0} - 1\right)\right]
\end{align}
where $\beta \geq 0$ is a dimensionless constant, $p_0$ represents the reference pressure (often taken as the atmospheric pressure, $p_{\mathrm{atm}}$), and $\mu_0$ denotes the reference viscosity, which is the viscosity of the fluid at the reference pressure. A comment on the notation: $\widehat{\mu}\big(p(\mathbf{x})\big)$ is a functional form representing viscosity in terms of $p(\mathbf{x})$. Different functional forms may arise by altering the argument while still representing the same viscosity (cf. Eq.~\eqref{Eqn:Hopf_Cole_Barus_model} and Eq.~\eqref{Eqn:Hopf_Cole_Modified_viscosity}).

With respect to boundary conditions, we divide the boundary into two complementary parts: $\Gamma^{p}$ and $\Gamma^{v}$. $\Gamma^{p}$ represents the part of the boundary on which pressure is prescribed, while $\Gamma^{v}$ denotes the part of the boundary on which the normal component of the velocity is prescribed. For the problem to be mathematically well-posed, we require that
%----------------------------;
%  Equation: Well-posedness  ;
%----------------------------;
\begin{align}
    \Gamma^{p} \cup \Gamma^{v} 
    = \partial \Omega \quad \mathrm{and} 
    \quad \Gamma^{p} \cap \Gamma^{v} 
    = \emptyset
\end{align}

The governing equations, accounting for the pressure dependence of the fluid's viscosity, are expressed as follows:
%---------------------------------;
%  Equation: Governing equations  ;
%---------------------------------;
\begin{subequations}
    \label{Eqn:Hopf_Cole_Original_BVP}
    \begin{alignat}{2}
        \label{Eqn:Hopf_Cole_BoLM}
        &\widehat{\mu}\big(p(\mathbf{x})\big) 
        \, \mathbf{K}^{-1}(\mathbf{x})
        \, \mathbf{v}(\mathbf{x}) 
        + \mathrm{grad}\big[p(\mathbf{x})\big] 
        = \rho \, \mathbf{b}(\mathbf{x}) 
        &&\quad \mathrm{in} \; \Omega \\
        \label{Eqn:Hopf_Cole_BoM}
        &\mathrm{div}\big[\mathbf{v}(\mathbf{x})\big] 
        = 0  
        &&\quad \mathrm{in} \; \Omega \\
        \label{Eqn:Hopf_Cole_pBC}
        &p(\mathbf{x}) = p_{\mathrm{p}}(\mathbf{x}) 
        &&\quad \mathrm{on} \; \Gamma^{p} \\
        \label{Eqn:Hopf_Cole_vBC}
        &\mathbf{v}(\mathbf{x}) \bullet 
        \widehat{\mathbf{n}}(\mathbf{x})  
        = v_{n}(\mathbf{x}) 
        &&\quad \mathrm{on} \; \Gamma^{v} 
    \end{alignat}
\end{subequations}
where $\mathbf{b}(\mathbf{x})$ is the specific body force field, $p_{\mathrm{p}}(\mathbf{x})$ the prescribed pressure on the boundary, and $v_n(\mathbf{x})$ the prescribed normal component of the velocity on the boundary

Equation \eqref{Eqn:Hopf_Cole_BoLM} represents the balance of linear momentum, Eq.~\eqref{Eqn:Hopf_Cole_BoM} accounts for the mass conservation considering the incompressibility of the fluid, Eq.~\eqref{Eqn:Hopf_Cole_pBC} describes the pressure boundary condition, and Eq.~\eqref{Eqn:Hopf_Cole_vBC} specifies the velocity boundary condition, with the normal component of the velocity prescribed.

The resulting mathematical model is nonlinear due to Eq.~\eqref{Eqn:Hopf_Cole_BoLM} in conjunction with Eq.~\eqref{Eqn:Hopf_Cole_Barus_model}. Specifically, Eqs.~\eqref{Eqn:Hopf_Cole_BoLM} and \eqref{Eqn:Hopf_Cole_BoM} can be combined to eliminate the velocity, resulting in a quasi-linear elliptic partial differential equation in terms of the pressure \citep{gilbarg1977elliptic}. As presented in the next section, we use the Hopf-Cole transformation to recast the mathematical model into a linear elliptic partial differential equation, expressed in terms of a transformed variable.
 
Before proceeding, a comment on mathematical consistency is warranted. If $\Gamma^{v} = \partial \Omega$ (i.e., pressure is not specified on any part of the boundary), then the prescribed velocity must satisfy the following compatibility equation:
%-------------------------------------;
%  Equation: Compatibility equation   ;
%-------------------------------------;
\begin{align}
    \label{Eqn:Hopf_Cole_Compatibility_equation}
    \int_{\Gamma^{v} = \partial \Omega} 
    v_n(\mathbf{x}) \, \mathrm{d} \Gamma = 0
\end{align}
The above condition is a direct consequence of the divergence theorem and Eq.~\eqref{Eqn:Hopf_Cole_BoM}. 

    %*********************************************;
%                                             ;
%  NAME                                       ;
%    S3_Hopf_Cole_Transformation.tex          ;
%                                             ;
%*********************************************;
\section{HOPF-COLE TRANSFORMATION}
\label{Sec:S3_Hopf_Cole_Transformation}
To avoid nonlinear terms after applying the Hopf-Cole transformation, some restrictions on the nature of the body force are necessary. Specifically, we assume that the body force is a conservative vector field---a reasonable assumption; for example, the popular body force due to gravity is a conservative field, too. Accordingly, we can write: 
%-------------------------------------;
%  Equation: Conservative body force  ;
%-------------------------------------;
\begin{align}
    \label{Eqn:Hopf_Cole_rhob_conservative}
    \rho \, \mathbf{b}(\mathbf{x}) 
    = -\mathrm{grad}\big[\xi(\mathbf{x})\big]
\end{align}
for some scalar field $\xi(\mathbf{x})$. For convenience, we introduce a modified pressure variable: 
%-------------------------------;
%  Equation: Ptilde definition  ;
%-------------------------------;
\begin{align}
    \label{Eqn:Hopf_Cole_Ptilde_definition}
    \widetilde{p}(\mathbf{x}) := 
    p(\mathbf{x}) + \xi(\mathbf{x}) 
\end{align}
According to formal definitions, $\widetilde{p}(\mathbf{x})$ represents the pressure relative to the datum specified by the negative of $\xi(\mathbf{x})$.

Using the above definition, we rewrite the Barus model as follows: 
%--------------------------------------;
%  Equation: mu_tilde functional form  ;
%--------------------------------------;
\begin{align}
   \label{Eqn:Hopf_Cole_mu_tilde_functional_form}
    \mu = \widetilde{\mu}\big(\mathbf{x},\widetilde{p}(\mathbf{x})\big) 
    := \mu_0 \, \exp\left[\beta \, \left(\frac{\widetilde{p}(\mathbf{x}) - \xi(\mathbf{x})}{p_0} - 1\right)\right]
\end{align}
indicating the functional dependence of $\widetilde{\mu}(\cdot)$, which depends on $\mathbf{x}$ both explicitly (due to the prescribed scalar field $\xi(\mathbf{x})$) and implicitly (due to the solution field $\widetilde{p}(\mathbf{x})$).
The exponential form of the Barus model allows us to rewrite the preceding equation using the following multiplicative decomposition: 
%------------------------------------------;
%  Equation: Multiplicative decomposition  ;
%------------------------------------------;
\begin{align}
    \mu = 
    \mu_0 \exp\left[-\frac{\beta \, 
    \xi(\mathbf{x})}{p_0}\right] 
    \exp\left[\beta \, 
    \left(\frac{\widetilde{p}(\mathbf{x})}{p_0} - 1\right) \right]
\end{align}
For future reference, we introduce the following notation: 
%------------------------------------------;
%  Equation: Notation for transformation  ;
%------------------------------------------;
\begin{subequations}
    \begin{align}
        \label{Eqn:Hopf_Cole_mu0_tilde}
        \widetilde{\mu}_{0}(\mathbf{x}) &:= 
        \mu_0 \exp\left[-\frac{\beta \, \xi(\mathbf{x})}{p_0}\right] \\ 
        \label{Eqn:Hopf_Cole_g_of_ptilde}
        g\big(\widetilde{p}(\mathbf{x})\big) &:= \exp\left[\beta \, \left(\frac{\widetilde{p}(\mathbf{x})}{p_0} - 1\right) \right]
    \end{align}
\end{subequations}
Then, the viscosity takes the following compact form: 
%--------------------------------;
%  Equation: Modified viscosity  ;
%--------------------------------;
\begin{align}
    \label{Eqn:Hopf_Cole_Modified_viscosity}
    \mu = \widetilde{\mu}\big(\mathbf{x},\widetilde{p}(\mathbf{x})\big)
    = \widetilde{\mu}_0(\mathbf{x}) \, g\big(\widetilde{p}(\mathbf{x})\big)
\end{align}
The indicated functional dependence of $\widetilde{\mu}$, $\widetilde{\mu}_0$, and $g$ will be crucial in recasting the governing equations.

Equations~\eqref{Eqn:Hopf_Cole_Ptilde_definition} and \eqref{Eqn:Hopf_Cole_Modified_viscosity} enable us to rewrite the balance of linear momentum \eqref{Eqn:Hopf_Cole_BoLM}, originally a non-homogeneous partial differential equation, as a homogeneous one. The boundary value problem written in terms of $\widetilde{p}(\mathbf{x})$ and $\mathbf{v}(\mathbf{x})$ takes the following form: 
%----------------------------------------------;
%  Equation: Governing equations conservative  ;
%----------------------------------------------;
\begin{subequations}
    \label{Eqn:Hopf_Cole_BVP_Conservative}
    \begin{alignat}{2}
        \label{Eqn:Hopf_Cole_BoLM_Conservative}
        &\widetilde{\mu}_{0}(\mathbf{x}) \, g\big(\widetilde{p}(\mathbf{x})\big) 
        \, \mathbf{K}^{-1}(\mathbf{x})
        \, \mathbf{v}(\mathbf{x}) 
        + \mathrm{grad}\big[\widetilde{p}(\mathbf{x})\big] 
        = \mathbf{0} 
        &&\quad \mathrm{in} \; \Omega \\
        \label{Eqn:Hopf_Cole_BoM_Conservative}
        &\mathrm{div}\big[\mathbf{v}(\mathbf{x})\big] 
        = 0  
        &&\quad \mathrm{in} \; \Omega \\
        \label{Eqn:Hopf_Cole_pBC_Conservative}
        &\widetilde{p}(\mathbf{x}) = \widetilde{p}_{\mathrm{p}}(\mathbf{x}) 
        &&\quad \mathrm{on} \; \Gamma^{p} \\
        \label{Eqn:Hopf_Cole_vBC_Conservative}
        &\mathbf{v}(\mathbf{x}) \bullet 
        \widehat{\mathbf{n}}(\mathbf{x})  
        = v_{n}(\mathbf{x}) 
        &&\quad \mathrm{on} \; \Gamma^{v} 
    \end{alignat}
\end{subequations}
where, in view of Eqs.~\eqref{Eqn:Hopf_Cole_pBC} and \eqref{Eqn:Hopf_Cole_Ptilde_definition}, the prescribed modified pressure takes the form:
%-----------------------;
%  Equation: p_tilde_p  ;
%-----------------------;
\begin{align}
    \label{Eqn:Hopf_Cole_Ptilde_p}
    \widetilde{p}_{\mathrm{p}}(\mathbf{x}) 
    := p_{\mathrm{p}}(\mathbf{x}) + \xi(\mathbf{x}) 
\end{align}

In the remainder of this paper, results will be obtained for $\widetilde{p}(\mathbf{x})$ by considering the boundary value problem given by Eqs.~\eqref{Eqn:Hopf_Cole_BoLM_Conservative}--\eqref{Eqn:Hopf_Cole_vBC_Conservative}, rather than addressing the original boundary value problem given by Eqs.~\eqref{Eqn:Hopf_Cole_BoLM}--\eqref{Eqn:Hopf_Cole_vBC}, which is formulated in terms of $p(\mathbf{x})$. Once $\widetilde{p}(\mathbf{x})$ is known, $p(\mathbf{x})$ can be calculated directly from Eq.~\eqref{Eqn:Hopf_Cole_Ptilde_definition}. The velocity field, however, remains unchanged under both the boundary value problems (cf. Eqs.~\eqref{Eqn:Hopf_Cole_BoLM}--\eqref{Eqn:Hopf_Cole_vBC} and Eqs.~\eqref{Eqn:Hopf_Cole_BoLM_Conservative}--\eqref{Eqn:Hopf_Cole_vBC_Conservative}). 

Nonetheless, Eq.~\eqref{Eqn:Hopf_Cole_BoLM_Conservative} still remains nonlinear due to the presence of $g\big(\widetilde{p}(\mathbf{x})\big)$. Appealing to the Hopf-Cole technique, we introduce a new field variable, $\mathcal{P}(\mathbf{x})$, defined by the following transformation:
%----------------------------;
%  Equation: Transformation  ;
%----------------------------;
\begin{align}
    \label{Eqn:Hopf_Cole_Transformed_variable}
    \widetilde{p}(\mathbf{x}) 
    = f\big(\mathcal{P}(\mathbf{x})\big)
\end{align}
We will choose a specific form of this transformation later. Applying the (spatial) gradient operator on both sides of the preceding equation, we write: 
%----------------------------------------------;
%  Equation: Gradient of transformed variable  ;
%----------------------------------------------;
\begin{align}
    \label{Eqn:Hopf_Cole_Gradient_of_trasformed_variable}
    \mathrm{grad}\big[\widetilde{p}(\mathbf{x})\big] 
    = \frac{\mathrm{d}f\big(\mathcal{P}\big)}{\mathrm{d}\mathcal{P}} \, 
    \mathrm{grad}\big[\mathcal{P}(\mathbf{x})\big]
\end{align}
Combining Eqs.~\eqref{Eqn:Hopf_Cole_Gradient_of_trasformed_variable} and  \eqref{Eqn:Hopf_Cole_BoLM_Conservative}, we arrive at:
%-------------------------------------------;
%  Equation: Velocity in terms of mathcalP  ;
%-------------------------------------------;
\begin{align}
    \label{Eqn:Hopf_Cole_v_in_terms_of_mathcalP}
    \mathbf{v}(\mathbf{x}) 
    &= - \Big( \widetilde{\mu}_0(\mathbf{x})\, g\big(\widetilde{p}(\mathbf{x})\big)\Big)^{-1} \,\mathbf{K}(\mathbf{x}) \, \text{grad}\big[\widetilde{p}(\mathbf{x})\big] 
    \nonumber \\
    &= -\left(\frac{1}{ g\Big(f\big(\mathcal{P}(\mathbf{x})\big)\Big)}\, \frac{\mathrm{d}f\big(\mathcal{P}\big)}{\mathrm{d}\mathcal{P}} \right) 
    \, \frac{1}{\widetilde{\mu}_0(\mathbf{x})} 
    \, \mathbf{K}(\mathbf{x}) \,  \text{grad}\big[\mathcal{P}(\mathbf{x})\big]
\end{align}
Invoking Eq.~\eqref{Eqn:Hopf_Cole_BoM} and the product rule for differentiation, we proceed as follows: 
%---------------------------------------;
%  Equation: BoLM after transformation  ;
%---------------------------------------;
\begin{align}
    \label{Eqn:Hopf_Cole_modified_nonlinear_BoLM}
    0 = \mathrm{div}\big[\mathbf{v}(\mathbf{x})\big] 
    &= -\text{div}\left[ \frac{1}{\widetilde{\mu}_0(\mathbf{x})} \,\mathbf{K}(\mathbf{x})\,\text{grad}\big[\mathcal{P}(\mathbf{x})\big] \right] \left(\underbrace{\frac{1}{ g\Big(f\big(\mathcal{P}(\mathbf{x})\big)\Big)}\, \frac{\mathrm{d}f\big(\mathcal{P}\big)}{\mathrm{d}\mathcal{P}}}_{\text{term 1}} \right) 
    \nonumber \\ 
    &\qquad \qquad - \frac{1}{\widetilde{\mu}_0(\mathbf{x})} \, \mathbf{K}(\mathbf{x})\, \big\| \text{grad}\big[\mathcal{P}(\mathbf{x})\big]\big\|^{2} 
    \left(\underbrace{\frac{\mathrm{d}}{\mathrm{d}\mathcal{P}} \left[\frac{1}{g\Big(f\big(\mathcal{P}(\mathbf{x})\big)\Big)} \frac{\mathrm{d}f\big(\mathcal{P}\big)}{\mathrm{d}\mathcal{P}}\right]}_{\text{term 2}}\right) 
\end{align}
where
%-----------------------------;
%  Equation: Norm of grad[p]  ;
%-----------------------------;
\begin{align}
    \big\| \text{grad}\big[\mathcal{P}(\mathbf{x})\big] \big\|^{2}
    =  \text{grad}\big[\mathcal{P}(\mathbf{x})\big] \bullet \text{grad}\big[\mathcal{P}(\mathbf{x})\big]
\end{align}
One way to eliminate nonlinearity in Eq.~\eqref{Eqn:Hopf_Cole_modified_nonlinear_BoLM} is to choose the function $f\big(\mathcal{P}(\mathbf{x})\big)$ such that term \#2 vanishes and term \#1 is non-zero. That is, 
%---------------------------;
%  Equation: Terms 1 and 2  ;
%---------------------------;
\begin{align}
    \label{Eqn:Hopf_Cole_term_2}
    &\frac{\mathrm{d}}{\mathrm{d}\mathcal{P}} \left[\frac{1}{g\Big(f\big(\mathcal{P}(\mathbf{x})\big)\Big)} \frac{\mathrm{d}f\big(\mathcal{P}\big)}{\mathrm{d}\mathcal{P}}\right] = 0 \\ 
    \label{Eqn:Hopf_Cole_term_1}
    &\frac{1}{ g\Big(f\big(\mathcal{P}(\mathbf{x})\big)\Big)}\, \frac{\mathrm{d}f\big(\mathcal{P}\big)}{\mathrm{d}\mathcal{P}} \neq 0 
\end{align}
Our approach to obtaining such a function $f\big(\mathcal{P}\big)$ is to seek a solution for Eq.~\eqref{Eqn:Hopf_Cole_term_2} and check whether this solution makes term \#1 non-zero (i.e., satisfies Eq.~\eqref{Eqn:Hopf_Cole_term_1}). To find the solution of Eq.~\eqref{Eqn:Hopf_Cole_term_2}, we integrate the equation once with respect to $\mathcal{P}$ and get:
%----------------------------------;
%  Equation: Arbitrary constant A  ;
%----------------------------------;
\begin{align}
    \label{Eqn:Hopf_Cole_Arbitrary_constant_A}
    \frac{1}{g\Big(f\big(\mathcal{P}(\mathbf{x})\big)\Big)} \, \frac{\mathrm{d}f\big(\mathcal{P}\big)}{\mathrm{d}\mathcal{P}} = A
\end{align}
where $A$ is an integration constant. By integrating the preceding equation once more, we arrive at: 
%----------------------------------;
%  Equation: Arbitrary constant B  ;
%----------------------------------;
\begin{align}
    \int \frac{1}{g\big(f\big)} \, \mathrm{d}f = A \, \mathcal{P} + B
\end{align}
with $B$ being another integration constant. Using the specific form of $g$ given by Eq.~\eqref{Eqn:Hopf_Cole_g_of_ptilde} and executing the above integral, we get:
%-----------------------------------;
%  Equation: f in terms of A and B  ;
%-----------------------------------;
\begin{align}
    \label{Eqn:Hopf_Cole_f_in_terms_of_A_and_B}
    -\frac{p_0}{\beta} \,\text{exp}\left[ -\beta \, 
    \left(\frac{f}{p_0} - 1\right) \right] 
    = A \, \mathcal{P} + B
\end{align}

Up to this point, the constants $A$ and $B$ remain arbitrary. We now choose 
%-------------------------------;
%  Equation: A and B constants  ;
%-------------------------------;
\begin{align}
    \label{Eqn:Hopf_Cole_selection_of_A_and_B}
    A = 1 \quad \mathrm{and} \quad B = 0 
\end{align}
Substituting the above constants into Eq.~\eqref{Eqn:Hopf_Cole_f_in_terms_of_A_and_B}, we obtain an explicit form of the function $f\big(\mathcal{P}(\mathbf{x})\big)$ and, thus, a mapping between $\widetilde{p}(\mathbf{x})$ and $\mathcal{P}(\mathbf{x})$: 
%----------------------------------------;
%  Equation: Explicit from of mathcal P  ;
%----------------------------------------;
\begin{align}
    \label{Eqn:Hopf_Cole_Functional_form_of_mathcal_P}
    \widetilde{p}(\mathbf{x}) 
    \equiv f\big(\mathcal{P}(\mathbf{x})\big) 
    = p_0 \left(1 - \frac{1}{\beta} \, \ln\left[-\frac{\beta\,\mathcal{P}(\mathbf{x})}{p_0}\right] \right) 
\end{align}
Figure \ref{Fig:Hopf_Cole_mathcal_P_graph} provides a visual representation of the mapping described above. The expression for the original pressure field, which is related to $\widetilde{p}(\mathbf{x})$ via Eq.~\eqref{Eqn:Hopf_Cole_Ptilde_definition}, reads: 
%-------------------------------------------;
%  Equation: Expression for pressure field  ;
%-------------------------------------------;
\begin{align}
    \label{Eqn:Hopf_Cole_Functional_form_of_original_p}
    p(\mathbf{x}) = -\xi(\mathbf{x}) + 
    p_0 \left(1 - \frac{1}{\beta} \, \ln\left[-\frac{\beta\,\mathcal{P}(\mathbf{x})}{p_0}\right] \right) 
\end{align}
Finally, for completeness, the inverse transformation, which takes $p(\mathbf{x})$ to $\mathcal{P}(\mathbf{x})$, can be explicitly written as:  
%------------------------------;
%  Equation: Transformation h  ;
%------------------------------;
\begin{align}
    \label{Eqn:Hopf_Cole_h_transformation}
    \mathcal{P}(\mathbf{x}) 
    \equiv h\big(\mathbf{x},p(\mathbf{x})\big) 
    := -\frac{p_0}{\beta} \, 
    \exp\left[-\beta \left(\frac{p(\mathbf{x}) 
    + \xi(\mathbf{x})}{p_0} - 1\right)\right] 
\end{align}

%-------------------------------------------;
%  Figure 1: Graphical representation of P  ;
%-------------------------------------------;
\begin{figure}[ht]
    \centering
    \includegraphics[scale=0.55]{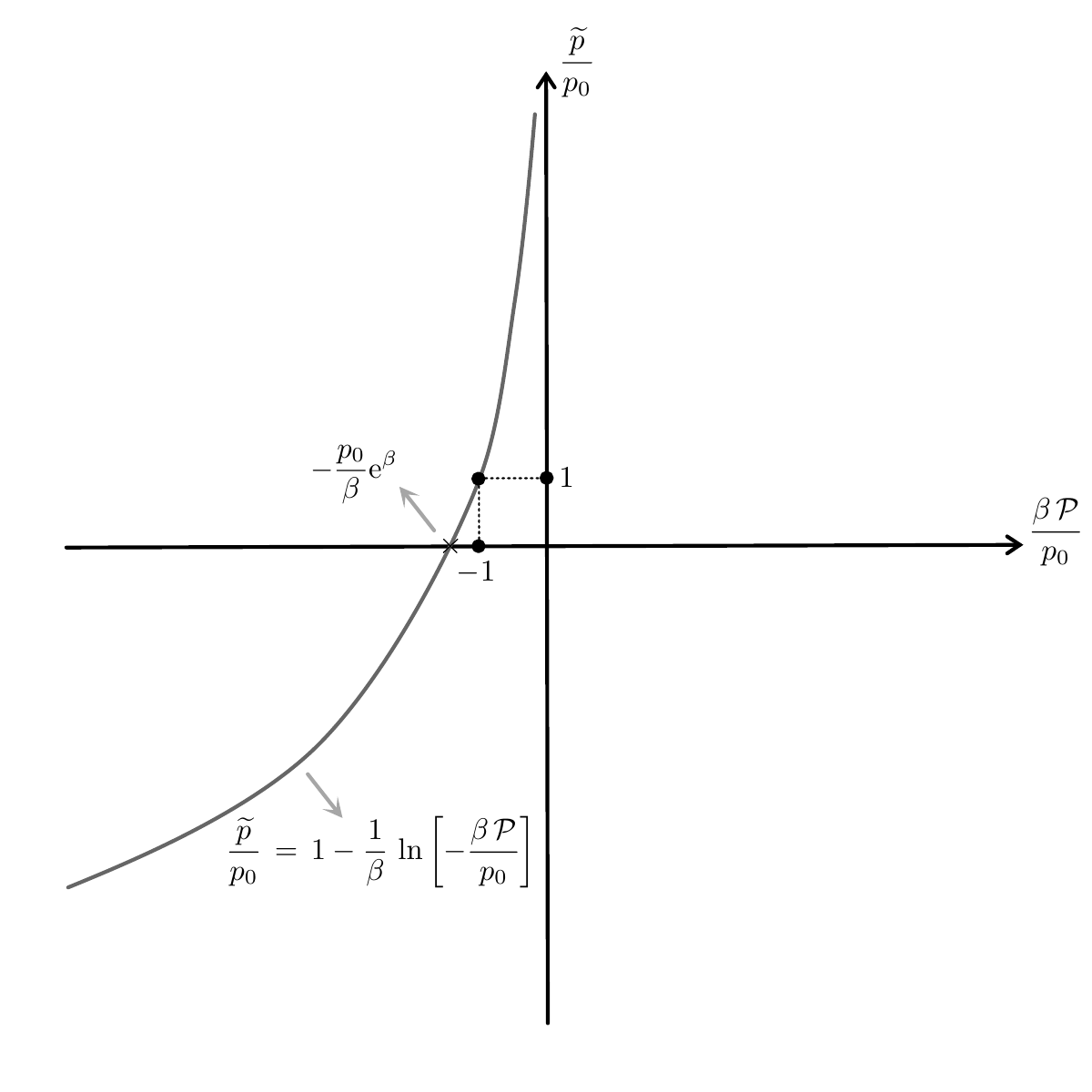}
    \caption{Graphical representation of $\frac{\tilde{p}}{p_0}$ vs. $\mathcal{P}$ based on Eq.~\eqref{Eqn:Hopf_Cole_Functional_form_of_mathcal_P}.}
    \label{Fig:Hopf_Cole_mathcal_P_graph}
\end{figure}

We now reformulate the boundary value problem \eqref{Eqn:Hopf_Cole_BoLM_Conservative}--\eqref{Eqn:Hopf_Cole_vBC_Conservative}, originally expressed in terms of $\widetilde{p}(\mathbf{x})$, using the newly introduced variable $\mathcal{P}(\mathbf{x})$. The transformed equations, written in terms of $\mathcal{P}(\mathbf{x})$ and $\mathbf{v}(\mathbf{x})$, take the following form:
%-----------------------------------;
%  Equation: Transformed equations  ;
%-----------------------------------;
\begin{subequations}
    \begin{alignat}{2}
       \label{Eqn:Hopf_Cole_Linear_BoLM}
        &\widetilde{\mu}_0(\mathbf{x}) \, \mathbf{K}^{-1}(\mathbf{x}) \, 
        \mathbf{v}(\mathbf{x}) 
        + \mathrm{grad}\big[\mathcal{P}(\mathbf{x})\big] = \mathbf{0} 
        &&\quad \mathrm{in} \; \Omega \\ 
        \label{Eqn:Hopf_Cole_Linear_BoM}
        &\mathrm{div}\big[\mathbf{v}(\mathbf{x})\big] = 0 
        &&\quad \mathrm{in} \; \Omega \\ 
        \label{Eqn:Hopf_Cole_Linear_pBC}
        &\mathcal{P}(\mathbf{x}) = \mathcal{P}_{\mathrm{p}}(\mathbf{x}) 
        &&\quad \mathrm{on} \; \Gamma^{p} \\ 
        \label{Eqn:Hopf_Cole_Linear_vBC}
        & \mathbf{v}(\mathbf{x}) \bullet \widehat{\mathbf{n}}(\mathbf{x}) = v_n(\mathbf{x}) 
        &&\quad \mathrm{on} \; \Gamma^{v} 
    \end{alignat}
\end{subequations}
where 
%-----------------------------------;
%  Equation: mathcalP_p definition  ;
%-----------------------------------;
\begin{align}
    \label{Eqn:Hopf_Cole_mathcalP_p_definition}
    \mathcal{P}_{\mathrm{p}}(\mathbf{x}) 
    = h\big(\mathbf{x},p_{\mathrm{p}}(\mathbf{x})\big)
\end{align}
with $h\big(\mathbf{x},p(\mathbf{x})\big)$ defined in Eq.~\eqref{Eqn:Hopf_Cole_h_transformation}. The above boundary value problem resembles that of the classical Darcy equations, which form a system of linear partial differential equations. 

%---------------------------------------;
%  Remark about the Kirchhoff approach  ;
%---------------------------------------;
\begin{remark}
    The Kirchhoff approach can be used to derive an alternative transformation that converts the nonlinear Barus model into the classical Darcy equations. In many cases, the transformation under the Kirchhoff method is more complex than the one derived using the Hopf-Cole method. Interestingly, the Kirchhoff transformation can be shown to be a special case of the Hopf-Cole approach. For a more detailed comparison of these two methods, refer to Appendix \ref{App:Hopf_Cole_Kirchhoff}.
\end{remark}

%=================================;
%  Subsection: Proposed approach  ;
%=================================;
\subsection{Proposed approach} The ultimate aim is to find a solution for the original nonlinear boundary value problem \eqref{Eqn:Hopf_Cole_BoLM}--\eqref{Eqn:Hopf_Cole_vBC}, given the boundary conditions $\widetilde{p}_{\mathrm{p}}(\mathbf{x})$ and $v_n(\mathbf{x})$. Instead of directly handling the nonlinearity, we propose a new approach based on the Hopf-Cole transformation.

The steps for solving a boundary value problem arising from the Barus model using the proposed approach are:
\begin{enumerate}
    \item[Step 1.]  Use the inverse transformation given by Eq.~\eqref{Eqn:Hopf_Cole_h_transformation} to compute the transformed boundary conditions: $\mathcal{P}_{\mathrm{p}}(\mathbf{x})$ and $v_n(\mathbf{x})$. Note that the velocity boundary conditions remain unaffected by the transformation. 
    \item[Step 2.] Solve the resulting linear boundary value problem given by Eqs.~\eqref{Eqn:Hopf_Cole_Linear_BoLM}--\eqref{Eqn:Hopf_Cole_Linear_vBC}. The output will be $\mathcal{P}(\mathbf{x})$ and $\mathbf{v}(\mathbf{x})$.
    \item[Step 3.] Apply the forward transformation given by Eq.~\eqref{Eqn:Hopf_Cole_Functional_form_of_original_p} to obtain the original pressure field $\widetilde{p}(\mathbf{x})$ from $\mathcal{P}(\mathbf{x})$. 
\end{enumerate}
Figure~\ref{Fig:Hopf_Cole_Flowchart} provides a pictorial summary of the work-flow for both the original and proposed approaches.

Since the proposed approach eliminates the nonlinearity in the boundary value problem \eqref{Eqn:Hopf_Cole_BoLM}--\eqref{Eqn:Hopf_Cole_vBC}, it offers a powerful alternative for obtaining numerical solutions without the need for a computationally expensive nonlinear solver. This advantage is further explored in Section \ref{Sec:S6_Hopf_Cole_NR}. Additionally, the approach can be leveraged to establish mathematical results for the original nonlinear boundary value problem, similar to many known results for the classical Darcy equations, which constitute a linear system.

%---------------------------------------------;
%  Figure 2: Flowchart of Hopf-Cole approach  ;
%---------------------------------------------;
\begin{figure}[h]
    \centering
    \includegraphics[width=0.65\linewidth]{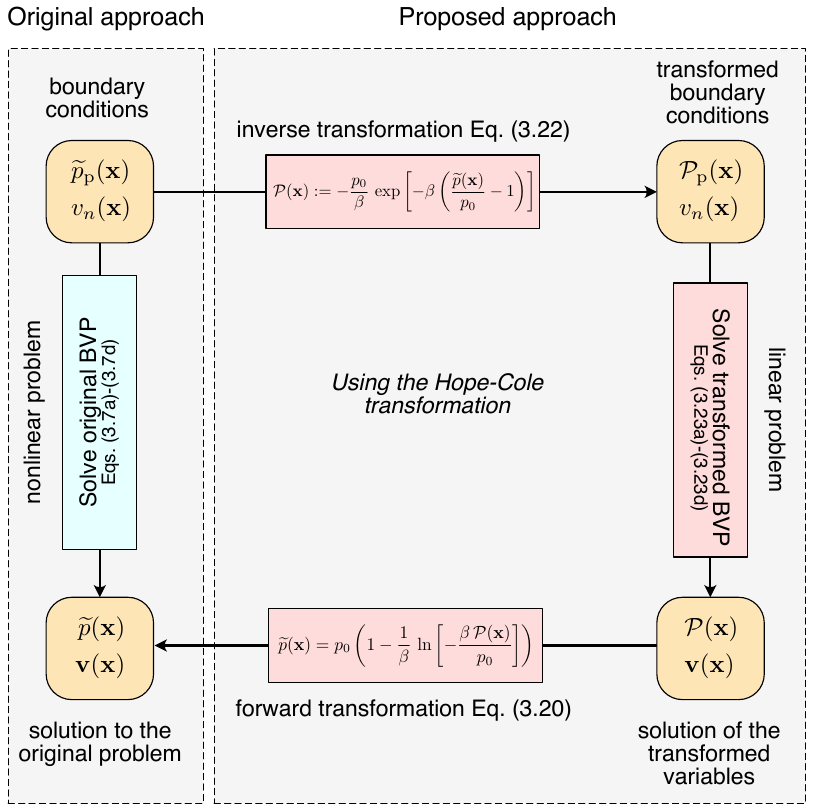}
    \caption{A flowchart showing the steps involved in obtaining the solution fields under the original and proposed approaches.
    \label{Fig:Hopf_Cole_Flowchart}}
\end{figure}

%===============================;
%  Subsection: Characteristics  ;
%===============================;
\subsection{Characteristics of the transformation}
To carry out the said study, we shall document some properties of the transformation $f\big(\mathcal{P}(\mathbf{x})\big)$, as given by Eq.~\eqref{Eqn:Hopf_Cole_Functional_form_of_mathcal_P}, which will play a crucial role. (Recall that Fig.~\ref{Fig:Hopf_Cole_mathcal_P_graph} offers a graphical representation of the transformation.)
\begin{enumerate}
    \item The domain of the transformation is $(-\infty,0)$, while the range is the entire real line. In other words, $\widetilde{p}(\mathbf{x})$ can take any real value, while the transformed variable $\mathcal{P}(\mathbf{x})$ can only take negative real values. 
    \item What is important, particularly when discussing the existence of solutions, is that a positive value for $\mathcal{P}(\mathbf{x})$ will not have a corresponding $\widetilde{p}(\mathbf{x})$, and therefore no corresponding real pressure $p(\mathbf{x})$. 
    \item Lastly, from the graph it is evident that the transformation from $\mathcal{P}(\mathbf{x})$ to $\widetilde{p}(\mathbf{x})$ is continuous, smooth, monotonically increasing, and bijective---hence invertible.
\end{enumerate} 

    %*********************************************;
%                                             ;
%  NAME                                       ;
%    S4_Hopf_Cole_Mathematical.tex            ;
%                                             ;
%*********************************************;
\section{MATHEMATICAL ANALYSIS}
\label{Sec:S4_Hopf_Cole_Mathematical}
In this section, we establish several qualitative mathematical properties of the nonlinear boundary value problem \eqref{Eqn:Hopf_Cole_BoLM}--\eqref{Eqn:Hopf_Cole_vBC}, including minimum and maximum principles, the comparison principle, and the uniqueness of solutions. These results can be derived using nonlinear analysis tools. For example, one can apply the comparison principle for quasilinear elliptic partial differential equations proposed by \citet{trudinger1974comparison} to the nonlinear boundary value problem. The other results will then follow from the comparison principle.

However, in this paper, we demonstrate that the proposed approach based on the Hopf-Cole transformation provides alternative and simpler proofs without relying on the theory of quasilinear partial differential equations. Our method exploits the properties of the Hopf-Cole transformation (i.e., Eq.~\eqref{Eqn:Hopf_Cole_Functional_form_of_mathcal_P}) and the fact that solutions of the transformed linear boundary value problem \eqref{Eqn:Hopf_Cole_Linear_BoLM}--\eqref{Eqn:Hopf_Cole_Linear_vBC} (which correspond to the classical Darcy equations) satisfy key properties, including minimum and maximum principles as well as the comparison principle.

These alternative proofs certainly have theoretical significance. In addition, the proofs are accessible to engineers, who often have exposure to the classical Darcy equations, their properties, and the rudiments of linear functional analysis. Thus, the presentation in this section has clear pedagogical value.

We conclude this section by outlining a class of problems for which solutions do not exist under the Barus model. This result is significant as it highlights the limitations of the Barus model's applicability---a point that has been overlooked in the literature.

%===============================================;
%  Subsection: Mathematical properties of P(x)  ;
%===============================================;
\subsection{Mathematical properties of $\mathcal{P}(\mathbf{x})$} 

We begin by establishing properties for the transformed variable, which satisfies the classical Darcy equations. To accomplish this, we recast the associated boundary value problem (given by Eqs.~\eqref{Eqn:Hopf_Cole_Linear_BoLM}--\eqref{Eqn:Hopf_Cole_Linear_vBC}) by eliminating $\mathbf{v}(\mathbf{x})$ and expressing it solely in terms of $\mathcal{P}(\mathbf{x})$. This reformulation converts the governing equations from mixed form to a second-order elliptic partial differential equation in a single variable: $\mathcal{P}(\mathbf{x})$. 

The reformulated boundary value problem that $\mathcal{P}(\mathbf{x})$ satisfies takes the following form:
%-------------------------------;
%  Equation: BVP for mathcal_P  ;
%-------------------------------;
\begin{subequations}
    \label{Eqn:Hopf_Cole_Pressure_Darcy_formulatiom}
    \begin{alignat}{2}
      \label{Eqn:Hopf_Cole_Pressure_BVP_GE}
        -&\mathrm{div}\left[ \frac{1}{\widetilde{\mu}_{0}(\mathbf{x})}\,\mathbf{K}(\mathbf{x})\, \text{grad}\big[\mathcal{P}(\mathbf{x})\big]\right] = 0 
        &&\quad \mathrm{in} \; \Omega \\ 
        \label{Eqn:Hopf_Cole_Pressure_BVP_pBC}
        &\mathcal{P}(\mathbf{x}) = \mathcal{P}_{\mathrm{p}}(\mathbf{x}) 
        &&\quad \mathrm{on} \; \Gamma^{p}\\
        \label{Eqn:Hopf_Cole_Pressure_BVP_vBC}
        -&\frac{1}{\widetilde{\mu}_{0}(\mathbf{x})}\,\mathbf{K}(\mathbf{x})\, \text{grad}\big[\mathcal{P}(\mathbf{x})\big] \bullet \widehat{\mathbf{n}}(\mathbf{x}) = v_{n}(\mathbf{x})
        &&\quad \mathrm{on} \; \Gamma^{v}
    \end{alignat}
\end{subequations}
Note that $\widetilde{\mu}_{0}(\mathbf{x}) > 0$.

Although the maximum and minimum principles can be established directly in the strong form (e.g., see \citep{gilbarg1977elliptic}), we instead adopt the technique of \citet{mudunuru2016enforcing} and use the weak form to prove these principles.

Accordingly, we introduce the following function spaces: 
%-----------------------------;
%  Equation: Function spaces  ;
%-----------------------------;
\begin{subequations}
    \begin{alignat}{2}
        \mathcal{U} & :=\left\{\mathcal{P}(\mathbf{x}) \in H^{1}(\Omega) \mid \text{trace}(\mathcal{P}(\mathbf{x}))=\mathcal{P}_{\mathrm{p}}(\mathbf{x}) \in L^{2}(\Gamma^{p}) \right\} \\
        \mathcal{W} & :=\left\{w(\mathbf{x}) \in H^{1}(\Omega) \mid \text{trace}(w(\mathbf{x})) = 0 \right\}
    \end{alignat}
\end{subequations}
where $H^{1}(\Omega)$ is a standard Sobolev space, $L^{2}(\Gamma^{p})$ is the space of square-integrable functions defined on $\Gamma^{p}$, and trace($\cdot$) denotes the standard trace operator from functional analysis, which assigns to a function its boundary values in a suitable functional space
\citep{evanspartial}. 

The weak form based on the Galerkin formalism reads: Find $\mathcal{P}(\mathbf{x}) \in \mathcal{U}$ such that we have 
%----------------------------------;
%  Equation: Galerkin formulation  ;
%----------------------------------;
\begin{align}
    \label{Eqn:Hopf_Cole_Galerkin_weak_form_pressure_1}
    \int_{\Omega} \text{grad}\big[w(\mathbf{x})\big] \, \bullet \, 
    \frac{1}{\widetilde{\mu}_{0}(\mathbf{x})}\, \mathbf{K}(\mathbf{x}) 
    \, \text{grad}\big[\mathcal{{P}}(\mathbf{x})\big] \, \mathrm{d}\Omega 
    + \int_{\Gamma^{v}} w(\mathbf{x}) \, v_{n}(\mathbf{x}) \, \mathrm{d}\Gamma  
    = 0 \quad \forall w(\mathbf{x}) \in \mathcal{W}
\end{align}

%======================================;
%  Theorem: Minimum principle (Darcy)  ;
%--------------------------------------;
\begin{lemma}[Minimum principle for $\mathcal{P}(\mathbf{x})$]
    \label{Lemma:Hopf_Cole_Minimum_principle_Darcy}
    Let $\mathcal{P}(\mathbf{x}) \in C^1(\Omega) \cap C^{0}(\overline{\Omega})$ be a solution of Galerkin weak formulation (i.e., Eq.~\eqref{Eqn:Hopf_Cole_Galerkin_weak_form_pressure_1}). If 
    \begin{align}
      \label{Eqn:BC_min}
      v_n(\mathbf{x}) \, \leq 0 \quad \text{on} \; \Gamma^{v} 
    \end{align}
    then $\mathcal{P}(\mathbf{x})$ satisfies the following lower bound:
    \begin{equation}
        \min_{\mathbf{x} \in \Gamma^{p}} \left[\mathcal{P}_\mathrm{p}(\mathbf{x})\right] \leq \mathcal{P}(\mathbf{x}) 
        \quad \forall \mathbf{x} \in \overline\Omega
    \end{equation}
    That is, the transformed field variable $\mathcal{P}(\mathbf{x})$ attains its minimum on the boundary where the pressure is prescribed. 
\end{lemma}

%======================================;
%  Proof of minimum principle (Darcy)  ;
%--------------------------------------;
\begin{proof}
    Let us introduce
    \begin{equation}
    \label{Eqn:Phi_min}
        \Phi_{\text{min}} =\min_{\mathbf{x} \in \Gamma^{p}} \left[\mathcal{P}_\mathrm{p}(\mathbf{x})\right] 
    \end{equation}
    We also define 
     \begin{equation}
     \label{Eqn:Eta_definition}
        \eta(\mathbf{x}) :=\min\big[ 0,\,  \mathcal{P}(\mathbf{x}) \, -\, \Phi_{\text{min}} \big]
    \end{equation}
    The newly introduced field variable $\eta(\mathbf{x})$ satisfies the following properties:
    \begin{enumerate}
    \item $\eta(\mathbf{x}) \, \in \, H^1(\Omega) \cap C^{0}(\overline{\Omega})$
    \item $\eta(\mathbf{x}) = 0 \quad \text{on}\; \Gamma^{p}$
    \item \label{prop:eta_condition}$\eta(\mathbf{x}) \leq 0 \quad \forall \mathbf{x}\, \in \,\overline{\Omega}$
    \item At any $\mathbf{x}$, either $\eta(\mathbf{x}) \, = \, 0 $ or $\eta(\mathbf{x}) \,+\, \Phi_{\text{min}}\, = \, \mathcal{P}(\mathbf{x})$
    \end{enumerate}

    Property \#2 implies that $\eta(\mathbf{x})$ can serve as a candidate for the weighting function. Setting $w(\mathbf{x}) = \eta(\mathbf{x})$ in the Galerkin weak formulation \eqref{Eqn:Hopf_Cole_Galerkin_weak_form_pressure_1}, we get the following: 
    %----------------------------------------;
    %  Equation: Galerkin weak form updated  ;
    %----------------------------------------;
    \begin{align}
    \label{Eqn:Hopf_Cole_Galerkin_weak_form_pressure_updated_1}
     \int_{\Omega} \text{grad}\big[\eta(\mathbf{x})\big] \bullet \frac{1}{\widetilde{\mu}_{0}(\mathbf{x})}\, \mathbf{K}(\mathbf{x}) \, \text{grad}\big[\mathcal{P}(\mathbf{x})\big] \, \mathrm{d}\Omega + \int_{\Gamma^{v}} \eta(\mathbf{x}) \, v_{n}(\mathbf{x}) \, \mathrm{d}\Gamma = 0 
    \end{align}
    Making use of Property \#4, we rewrite the preceding equation as follows:
    \begin{align}
     \label{Eqn:Galerkin_weak_form_pressure_updated_2}
       \int_{\Omega} \text{grad}\big[\eta(\mathbf{x})\big] \bullet  \frac{1}{\widetilde{\mu}_{0}(\mathbf{x})}\, \mathbf{K}(\mathbf{x}) \, \text{grad}\big[\eta(\mathbf{x}) +  \Phi_{\text{min}}\big] \, \mathrm{d}\Omega +  
       \int_{\Gamma^{v}} \eta(\mathbf{x}) \, v_{n}(\mathbf{x}) \, \mathrm{d}\Gamma = 0 
    \end{align}
    Noting that $\Phi_{\mathrm{min}}$ is a constant, we arrive at the following:
    \begin{align}
     \label{Eqn:Galerkin_weak_form_pressure_updated_3}
       \int_{\Omega} \text{grad}\big[\eta(\mathbf{x})\big] \bullet \frac{1}{\widetilde{\mu}_{0}(\mathbf{x})}\, \mathbf{K}(\mathbf{x}) \, \text{grad}\big[\eta(\mathbf{x})\big] \, \mathrm{d}\Omega + \int_{\Gamma^{v}} \eta(\mathbf{x}) \, v_{n}(\mathbf{x}) \, \mathrm{d}\Gamma 
       = 0 
    \end{align}
    
    Since the permeability tensor $\mathbf{K}(\mathbf{x})$ is positive definite (in fact, it is uniformly elliptic) and $\widetilde{\mu}_0(\mathbf{x}) > 0$, the first integral in the above equation is non-negative. As $\eta(\mathbf{x}) \leq 0$ (i.e., Property \#3) and $v_n(\mathbf{x}) \leq 0$, the second integral is also non-negative. Therefore, we conclude that both these integrals should vanish separately: 
    \begin{align}
        \label{Eqn:Hopf_Cole_semi_norm}
       &\int_{\Omega} \text{grad}\big[\eta(\mathbf{x})\big] \, \bullet \, \frac{1}{\widetilde{\mu}_{0}(\mathbf{x})}\, \mathbf{K}(\mathbf{x}) \, \text{grad}\big[\eta(\mathbf{x})\big] \, \mathrm{d}\Omega = 0 \\
       \label{Eqn:Hopf_Cole_eta_term_on_gammav}
       &\int_{\Gamma^{v}} \eta(\mathbf{x}) \, v_{n}(\mathbf{x}) \, \mathrm{d}\Gamma = 0 
    \end{align}
    
    As the integral in Eq.~\eqref{Eqn:Hopf_Cole_semi_norm} is the square of a semi-norm---to be precise, it is a weighted $H^1$ semi-norm with weight $\frac{1}{\widetilde{\mu}_{0}(\mathbf{x})}\mathbf{K}(\mathbf{x})$---and $\eta(\mathbf{x}) \in C^{0}(\overline{\Omega})$, we assert that 
    \begin{equation}
        \label{Eqn:Cole_Hopf_eta_constant_in_Omega}
        \eta(\mathbf{x}) = \mathrm{constant} \quad \forall \mathbf{x} \in \overline{\Omega}
    \end{equation}
    Eq.~\eqref{Eqn:Hopf_Cole_eta_term_on_gammav} and $\eta(\mathbf{x}) = 0$ on $\Gamma^{p}$ (Property \#2) together imply that 
    \begin{align}
        \label{Eqn:Cole_Hopf_eta_constant_on_partial_Omega}
        \eta(\mathbf{x}) = 0 \quad \mathrm{on} \; \partial \Omega 
    \end{align}
    Eqs.~\eqref{Eqn:Cole_Hopf_eta_constant_in_Omega} and \eqref{Eqn:Cole_Hopf_eta_constant_on_partial_Omega} alongside $\eta(\mathbf{x}) \in C^{0}(\overline{\Omega})$ establishes the following:  
    \begin{equation}
        \eta(\mathbf{x}) \, = \, 0 \quad \forall \mathbf{x} \in \overline{\Omega}
    \end{equation}
    Finally, using Eq.~\eqref{Eqn:Eta_definition}, we conclude:
    \begin{equation}
        \Phi_{\text{min}} \,\leq \,\mathcal{P}(\mathbf{x})
    \end{equation}
    which, based on the definition of $\Phi_{\text{min}}$ provided in Eq.~\eqref{Eqn:Phi_min}, yields the intended result:
    \begin{equation}
        \min_{\mathbf{x} \in \Gamma^{p}} \left[\mathcal{P}_\mathrm{p}(\mathbf{x})\right] \leq \mathcal{P}(\mathbf{x})  \quad \forall \mathbf{x} \in \overline{\Omega}
    \end{equation}
    %%%
\end{proof}

%============================;
%  Lemma: Maximum principle  ;
%----------------------------;
\begin{lemma}[Maximum principle for $\mathcal{P}(\mathbf{x})$]
    \label{Lemma:Hopf_Cole_Maximum_principle_Darcy}
    Let $\mathcal{P}(\mathbf{x}) \in C^1(\Omega) \cap C^{0}(\overline{\Omega})$ be a solution of Galerkin weak formulation (i.e., Eq.~\eqref{Eqn:Hopf_Cole_Galerkin_weak_form_pressure_1}). If 
    \begin{align}
      \label{Eqn:BC_max}
      v_n(\mathbf{x}) \, \geq 0 \quad \text{on} \; \Gamma^{v} 
    \end{align}
    then $\mathcal{P}(\mathbf{x})$ satisfies the following upper bound:
    \begin{align}
    \mathcal{P}(\mathbf{x}) \leq 
    \max_{\mathbf{x} \in \Gamma^{p}} \mathcal{P}_\mathrm{p}(\mathbf{x})
    \quad \forall \mathbf{x} \in \overline{\Omega}
    \end{align}
    That is, the transformed field variable $\mathcal{P}(\mathbf{x})$ attains its maximum on the boundary where the pressure is prescribed. 
\end{lemma}
%======================================;
%  Proof of lemma (Maximum principle)  ;
%--------------------------------------;
\begin{proof}
    Let us consider
    \begin{align}
    \mathcal{P}^{\#}(\mathbf{x}) 
    = -\mathcal{P}(\mathbf{x})
    \end{align}
    The newly introduced field variable satisfies, in weak or distributional sense, the following boundary value problem:
    \begin{subequations}
        \begin{alignat}{2}
            -&\mathrm{div}\left[ \frac{1}{\widetilde{\mu}_{0}(\mathbf{x})} 
            \,\mathbf{K}(\mathbf{x}) \, 
            \text{grad}\big[\mathcal{P}^{\#}(\mathbf{x})\big]\right] = 0 
            &&\quad \mathrm{in} \; \Omega \\ 
            &\mathcal{P}^{\#}(\mathbf{x})
            = \mathcal{P}^{\#}_{\mathrm{p}}(\mathbf{x}) 
            \equiv -\mathcal{P}_{\mathrm{p}}(\mathbf{x}) 
            &&\quad \mathrm{on} \; \Gamma^{p} \\
            -&\frac{1}{\widetilde{\mu}_{0}(\mathbf{x})}\,\mathbf{K}(\mathbf{x})\, \text{grad}\big[\mathcal{P}^{\#}(\mathbf{x})\big] \bullet 
            \widehat{\mathbf{n}}(\mathbf{x}) 
            = v_{n}^{\#}(\mathbf{x}) 
            \equiv -v_{n}(\mathbf{x})
            &&\quad \mathrm{on} \; \Gamma^{v}
     \end{alignat}
    \end{subequations}
    Noting that $v_n^{\#}(\mathbf{x}) \leq 0$ (as $v_n(\mathbf{x}) \geq 0$) and using the minimum principle given by Lemma \ref{Lemma:Hopf_Cole_Minimum_principle_Darcy}, we establish the following inequality: 
     \begin{equation}
        \min_{\mathbf{x} \in \Gamma^{p}} \left[\mathcal{P}^{\#}_\mathrm{p}(\mathbf{x})\right] \leq \mathcal{P}^{\#}(\mathbf{x})
        \quad \forall \mathbf{x} \in \overline\Omega
    \end{equation}
    Invoking that $\mathcal{P}^{\#}_{\mathrm{p}}(\mathbf{x}) = -\mathcal{P}_{\mathrm{p}}(\mathbf{x})$ and $\mathcal{P}^{\#}(\mathbf{x}) = -\mathcal{P}(\mathbf{x})$, we get the required inequality: 
    \begin{equation}
        \mathcal{P}(\mathbf{x}) \leq 
        \max_{\mathbf{x} \in \Gamma^{p}} \left[\mathcal{P}_\mathrm{p}(\mathbf{x})\right] 
        \quad \forall \mathbf{x} \in \overline\Omega
    \end{equation}
    %%%
\end{proof}

%================================;
%  Lemma: Comparison principle  ;
%--------------------------------;
\begin{lemma}[Comparison principle for $\mathcal{P}(\mathbf{x})$]
  \label{Lemma:Hopf_Cole_Comparision_principle_Darcy}
    Let $\big(\mathbf{v}^{(1)}(\mathbf{x}), \mathcal{P}^{(1)}(\mathbf{x})\big)$ and $\left(\mathbf{v}^{(2)}(\mathbf{x}), \mathcal{P}^{(2)}(\mathbf{x})\right)$ be the solutions of Eqs.~\eqref{Eqn:Hopf_Cole_Linear_BoLM}--\eqref{Eqn:Hopf_Cole_Linear_vBC} under the boundary conditions $\left(v_{n}^{(1)}(\mathbf{x}), \mathcal{P}_{\mathrm{p}}^{(1)}(\mathbf{x})\right)$ and $\left(v_{n}^{(2)}(\mathbf{x}), \mathcal{P}_{\mathrm{p}}^{(2)}(\mathbf{x})\right)$, respectively. If 
    %-------------------------------------------------;
    %  Equation: Conditions for comparison principle  ;
    %-------------------------------------------------;
    \begin{align}
        v^{(1)}_{n}(\mathbf{x}) \geq v^{(2)}_{n}(\mathbf{x}) 
        \quad \forall \mathbf{x} \in \Gamma^{v} 
        \quad \mathrm{and} \quad 
        \mathcal{P}_{\mathrm{p}}^{(2)}(\mathbf{x}) \geq  \mathcal{P}_{\mathrm{p}}^{(1)}(\mathbf{x})
        \quad \forall \mathbf{x} \in \Gamma^{p} 
    \end{align}
    then 
    \begin{align}
        \mathcal{P}^{(2)}(\mathbf{x}) \geq \mathcal{P}^{(1)}(\mathbf{x}) \quad \forall \mathbf{x}\in\overline{\Omega}
    \end{align}    
    %%%
\end{lemma}
%=========================================;
%  Proof of lemma (Comparison principle)  ;
%-----------------------------------------;
\begin{proof}
    Consider the following differences: 
    \begin{align}
        \label{Eqn:Hopf_Cole_Difference_mathcalP}
        &\mathcal{P}^{\star}(\mathbf{x}) = \mathcal{P}^{(2)}(\mathbf{x}) - \mathcal{P}^{(1)}(\mathbf{x}) \\
        \label{Eqn:Hopf_Cole_Difference_v}
        &\mathbf{v}^{\star}(\mathbf{x}) = \mathbf{v}^{(2)}(\mathbf{x}) - \mathbf{v}^{(2)}(\mathbf{x}) 
    \end{align}
    The above differences satisfy the following boundary value problem, in weak or distributional sense:
    %-----------------------------------;
    %  Equation: Transformed equations  ;
    %-----------------------------------;
    \begin{subequations}
        \begin{alignat}{2}
          \label{Eqn:Hopf_Cole_Linear_BoLM_CP}
          &\widetilde{\mu}_0(\mathbf{x}) \, \mathbf{K}^{-1}(\mathbf{x}) \, 
          \mathbf{v}^{\star}(\mathbf{x}) 
           + \mathrm{grad}\big[\mathcal{P}^{\star}(\mathbf{x})\big] = \mathbf{0} 
          &&\quad \mathrm{in} \; \Omega \\ 
          \label{Eqn:Hopf_Cole_Linear_BoM_CP}
           &\mathrm{div}\big[\mathbf{v}^{\star}(\mathbf{x})\big] = 0 
           &&\quad \mathrm{in} \; \Omega \\ 
           \label{Eqn:Hopf_Cole_Linear_pBC_CP}
           &\mathcal{P}^{\star}(\mathbf{x}) = \mathcal{P}^{\star}_{\mathrm{p}}(\mathbf{x})
           = \mathcal{P}^{(2)}_{\mathrm{p}}(\mathbf{x}) - \mathcal{P}^{(1)}_{\mathrm{p}}(\mathbf{x})
           \geq 0 
           &&\quad \mathrm{on} \; \Gamma^{p} \\ 
           \label{Eqn:Hopf_Cole_Linear_vBC_CP}
           &\mathbf{v}^{\star}(\mathbf{x}) \bullet \widehat{\mathbf{n}}(\mathbf{x}) = v^{\star}_n(\mathbf{x}) 
           = v^{(2)}_n(\mathbf{x}) - v^{(1)}_n(\mathbf{x}) 
           \leq 0 
           &&\quad \mathrm{on} \; \Gamma^{v} 
         \end{alignat}
    \end{subequations}
    
    Because $\mathcal{P}^{(1)}(\mathbf{x})$ and $\mathcal{P}^{(2)}(\mathbf{x})$ belong to $C^{1}(\Omega) \cap C^{0}(\overline{\Omega})$, the difference $\mathcal{P}^{\star}(\mathbf{x})$ too belongs to $C^{1}(\Omega) \cap C^{0}(\overline{\Omega})$. Appealing to the minimum principle given by Lemma \ref{Lemma:Hopf_Cole_Minimum_principle_Darcy},  we arrive at the following:
    \begin{equation}
        \min_{\mathbf{x} \in \Gamma^{p}} \left[\mathcal{P}^{\star}_\mathrm{p}(\mathbf{x})\right] \leq \mathcal{P}^{\star}(\mathbf{x})  \quad \forall \mathbf{x} \in \overline{\Omega}
    \end{equation}
    Since $\mathcal{P}^{\star}_{\mathrm{p}}(\mathbf{x}) \geq 0$ by the virtue of Eq.~\eqref{Eqn:Hopf_Cole_Linear_pBC}, we conclude that 
    \begin{align}
        0 \leq \min_{\mathbf{x} \in \Gamma^{p}} \left[\mathcal{P}^{\star}_\mathrm{p}(\mathbf{x})\right]
    \end{align}
    Combining the preceding two inequalities, we get the desired inequality:
    \begin{align}
        0 \leq \mathcal{P}_{\mathrm{p}}^{\star}(\mathbf{x}) = \mathcal{P}^{(2)}(\mathbf{x}) - \mathcal{P}^{(1)}(\mathbf{x}) \quad \forall\mathbf{x}\in \overline{\Omega}
    \end{align}
    %%%
\end{proof}

%==================================;
%  Lemma: Uniqueness of solutions  ;
%----------------------------------;
\begin{lemma}[Uniqueness of solutions for $\mathcal{P}(\mathbf{x})$]
   \label{Lemma:Hopf_Cole_Uniqueness_Darcy}
    The solutions to the boundary value problem \eqref{Eqn:Hopf_Cole_Linear_BoLM}--\eqref{Eqn:Hopf_Cole_Linear_vBC} are unique.  
\end{lemma}
%--------------------;
%  Uniqueness proof  ;
%--------------------;
\begin{proof} On the contrary, assume that there exist two solutions, labeled:
\[
\big(\mathcal{P}^{(1)}(\mathbf{x}),\mathbf{v}^{(1)}(\mathbf{x})\big) 
\quad \mathrm{and} \quad \big(\mathcal{P}^{(2)}(\mathbf{x}),\mathbf{v}^{(2)}(\mathbf{x})\big) 
\]
Just as in Lemma \ref{Lemma:Hopf_Cole_Comparision_principle_Darcy}, we again consider the differences defined in Eqs.~\eqref{Eqn:Hopf_Cole_Difference_mathcalP} and \eqref{Eqn:Hopf_Cole_Difference_v}. These differences satisfy the following boundary value problem (should be interpreted in distributional sense):
    %-----------------------------------;
    %  Equation: Transformed equations  ;
    %-----------------------------------;
    \begin{subequations}
        \begin{alignat}{2}
          \label{Eqn:Hopf_Cole_Uniqueness_BoLM}
          &\widetilde{\mu}_0(\mathbf{x}) \, \mathbf{K}^{-1}(\mathbf{x}) \, 
          \mathbf{v}^{\star}(\mathbf{x}) 
           + \mathrm{grad}\big[\mathcal{P}^{\star}(\mathbf{x})\big] = \mathbf{0} 
          &&\quad \mathrm{in} \; \Omega \\ 
          \label{Eqn:Hopf_Cole_Uniqueness_BoM}
           &\mathrm{div}\big[\mathbf{v}^{\star}(\mathbf{x})\big] = 0 
           &&\quad \mathrm{in} \; \Omega \\ 
           \label{Eqn:Hopf_Cole_Uniqueness_pBC}
           &\mathcal{P}^{\star}(\mathbf{x}) = \mathcal{P}^{\star}_{\mathrm{p}}(\mathbf{x})
           = 0 
           &&\quad \mathrm{on} \; \Gamma^{p} \\ 
           \label{Eqn:Hopf_Cole_Uniqueness_vBC}
           &\mathbf{v}^{\star}(\mathbf{x}) \bullet \widehat{\mathbf{n}}(\mathbf{x}) 
           = v_n(\mathbf{x}) = 0 
           &&\quad \mathrm{on} \; \Gamma^{v} 
         \end{alignat}
    \end{subequations}
    Applying the minimum principle and the maximum principle, we establish the following: 
    \begin{align}
      \label{Eqn:Hopf_Cole_pressure_diff_max_min}
        0 = \min_{\mathbf{x} \in \Gamma^{p}} \big[\mathcal{P}^{\star}_{\mathrm{p}}(\mathbf{x})\big] \leq \mathcal{P}^{\star}(\mathbf{x}) \leq 
        \max_{\mathbf{x} \in \Gamma^{p}} \big[\mathcal{P}^{\star}_{\mathrm{p}}(\mathbf{x})\big] = 0 
    \end{align}
    which further implies the coincidence of the pressure fields: 
    \begin{align}
        0 = \mathcal{P}^{\star}(\mathbf{x}) = 
        \mathcal{P}^{(2)}(\mathbf{x}) - 
        \mathcal{P}^{(1)}(\mathbf{x})
    \end{align}
    Then Eq.~\eqref{Eqn:Hopf_Cole_Uniqueness_BoLM} entails that 
    \begin{align}
        \widetilde{\mu}_0(\mathbf{x}) \, \mathbf{K}^{-1}(\mathbf{x}) \, \big(\mathbf{v}^{(1)}(\mathbf{x}) - \mathbf{v}^{(2)}(\mathbf{x}) \big) = \mathbf{0} 
    \end{align}
    Since $\widetilde{\mu}_0(\mathbf{x}) > 0$ and $\mathbf{K}(\mathbf{x})$ is invertible, the preceding equation yields: 
    \begin{align}
        \mathbf{v}^{(1)}(\mathbf{x}) = \mathbf{v}^{(2)}(\mathbf{x}) 
    \end{align}
    Hence, both solutions coincide. 
    %%%
    \end{proof}

%==================================================;
%  Subsection: Mathematical properties of p_tilde  ;
%==================================================;
\subsection{Mathematical properties of $\widetilde{p}(\mathbf{x})$} Recall that $\widetilde{p}(\mathbf{x})$ is the solution under the Barus model, considering that the body force is conservative. Leveraging the key features of the transformation given by Eq.~\eqref{Eqn:Hopf_Cole_Functional_form_of_mathcal_P}---particularly its invertibility and monotonicity---we extend the properties satisfied by $\mathcal{P}(\mathbf{x})$, as established in the preceding section, to $\widetilde{p}(\mathbf{x})$. We start with the minimum principle for $\widetilde{p}(\mathbf{x})$.

%======================================;
%  Theorem: Minimum principle (Barus)  ;
%--------------------------------------;
\begin{theorem}[Minimum principle under the Barus model]
    \label{Thm:Hopf_Cole_Minimum_principle_Barus}
    Let $\widetilde{p}(\mathbf{x}) \in C^{1}(\Omega) \cap C^{0} (\overline{\Omega})$. If $v_{n}(\mathbf{x})\, \leq \, 0 $ on $\Gamma^{v}$ then 
    \begin{align}
        \min_{\mathbf{x} \in \Gamma^p}
        \big[\widetilde{p}_{\mathrm{p}}(\mathbf{x})\big]
        \leq \widetilde{p}(\mathbf{x}) 
        \quad \forall \mathbf{x}\in \overline{\Omega}
\end{align}
\end{theorem}

%======================================;
%  Proof of minimum principle (Barus)  ;
%--------------------------------------;
\begin{proof}
    First, we note that the transformation that maps $\widetilde{p}(\mathbf{x})$ to $\mathcal{P}(\mathbf{x})$, which takes the following form: 
    \begin{align}
        \label{Eqn:Hopf_Cole_Functional_form_of_p}
        \mathcal{P}(\mathbf{x}) = -\frac{p_{0}}{\beta}\, 
        \text{exp}\left[ -\beta \left(\frac{\widetilde{p}(\mathbf{x})}{p_0} - 1 \right) \right]
    \end{align}
    is continuous. Consequently, $\widetilde{p}(\mathbf{x}) \in C^{0}(\overline{\Omega})$ implies that $\mathcal{P}(\mathbf{x}) \in C^{0}(\overline{\Omega})$, as a composite function of two continuous functions is continuous. 
    
    Second, since $\widetilde{p}(\mathbf{x}) \in C^{1}(\Omega)$, it is differentiable. Moreover, the above transformation \eqref{Eqn:Hopf_Cole_Functional_form_of_p} is differentiable. Appealing to the chain rule: composition of two differentiable functions is differentiable \citep{bartle2000introduction}, we conclude that $\mathcal{P}(\mathbf{x}) \in C^1(\Omega)$. Therefore, we establish that 
    \begin{align}
        \mathcal{P}(\mathbf{x}) \in C^{1}(\Omega) \cap C^{0}(\overline{\Omega})
    \end{align}

    The domain over which the exponential mapping is defined is the whole of $\mathbb{R}$. So, any value of $\widetilde{p}_{\mathrm{p}}(\mathbf{x})$ will render a well-defined $\mathcal{P}_{\mathrm{p}}(\mathbf{x})$ on $\Gamma^{p}$, via the above transformation \eqref{Eqn:Hopf_Cole_Functional_form_of_p}. 
    
    Given that $\mathcal{P}(\mathbf{x})$ satisfies all the requirements of Lemma~\ref{Lemma:Hopf_Cole_Minimum_principle_Darcy}, and noting that $v_n(\mathbf{x}) \leq 0$ on $\Gamma^{v}$, we arrive at the following inequality:
    \begin{equation}
         \min_{\mathbf{x} \in \Gamma^{p}} \left[\mathcal{P}_\mathrm{p}(\mathbf{x})\right] \leq \mathcal{P}(\mathbf{x}) 
         \quad \forall \mathbf{x} \in \overline\Omega
    \end{equation}
Since the transformation Eq.~\eqref{Eqn:Hopf_Cole_Functional_form_of_mathcal_P} is invertible---assuming that $\mathcal{P}(\mathbf{x}) \in (-\infty,0)$; otherwise solution does not exist; see \S\ref{Sec:Cole_Hopf_Ill_posedness}---and monotonically increasing, we establish the desired result:
\begin{align}
    \min_{\mathbf{x} \in \Gamma^p}
    \left[\widetilde{p}_{\mathrm{p}}(\mathbf{x})\right]
    \;\leq\;
    \widetilde{p}(\mathbf{x})
    \quad
    \forall\mathbf{x}\in \overline{\Omega}
\end{align}
%%%
\end{proof}

%======================================;
%  Theorem: Maximum principle (Barus)  ;
%--------------------------------------;
\begin{theorem}[Maximum principle under the Barus model]
    \label{Thm:Hopf_Cole_Maximum_principle_Barus}
    Let $\widetilde{p}(\mathbf{x}) \in C^{1}(\Omega) \cap C^{0} (\overline{\Omega})$. If $v_{n}(\mathbf{x})\, \geq \, 0 $ on $\Gamma^{v}$ then 
    \begin{align} 
        \max_{\mathbf{x} \in \Gamma^p}
        \left[\widetilde{p}_{\mathrm{p}}(\mathbf{x})\right]
        \geq \widetilde{p}(\mathbf{x}) 
        \quad \forall \mathbf{x}\in \overline{\Omega}
    \end{align}
\end{theorem}
%------------------------------------;
%  Proof: Maximum principle (Barus)  ;
%------------------------------------;
\begin{proof}
     Many of the steps follow the preceding theorem. The transformed field variable $\mathcal{P}(\mathbf{x})$, defined via Eq.~\eqref{Eqn:Hopf_Cole_Functional_form_of_p}, satisfy the requirements of Lemma \ref{Lemma:Hopf_Cole_Maximum_principle_Darcy}. Noting that $v_n(\mathbf{x}) \geq 0$ on $\Gamma^{v}$, the maximum principle for $\mathcal{P}(\mathbf{x})$ (i.e., Lemma \ref{Lemma:Hopf_Cole_Maximum_principle_Darcy}) implies:
    %------------------------------;
    %  Equation: Invoking Lemma 3  ;
    %------------------------------;
    \begin{align}
        \mathcal{P}(\mathbf{x}) \leq 
        \max_{\mathbf{x} \in \Gamma^{p}} \mathcal{P}_\mathrm{p}(\mathbf{x}),
        \quad \forall \mathbf{x} \in \overline{\Omega}.
    \end{align}
    Just like the preceding theorem, the bijectivity and monotonicity of the transformation---given by Eq.~\eqref{Eqn:Hopf_Cole_Functional_form_of_p}---provide the desired result:
    %---------------------------------------------;
    %  Equation: Final form of maximum principle  ;
    %---------------------------------------------;
    \begin{align} 
        \max_{\mathbf{x} \in \Gamma^p}
        \left[\widetilde{p}_{\mathrm{p}}(\mathbf{x})\right]
        \geq \widetilde{p}(\mathbf{x}) 
        \quad \forall \mathbf{x}\in \overline{\Omega}
    \end{align}    
    %%%
\end{proof}

%=========================================;
%  Theorem: Comparison principle (Barus)  ;
%-----------------------------------------;
\begin{theorem}[Comparison principle under the Barus model]
    \label{Thm:Hopf_Cole_comparison_principle_Barus}
    Let $\left( \mathbf{v}^{(1)}(\mathbf{x}), \widetilde{p}^{(1)}(\mathbf{x})\right)$ and $\left( \mathbf{v}^{(2)}(\mathbf{x}), \widetilde{p}^{(2)}(\mathbf{x})\right)$ be two solutions of the boundary value problem \eqref{Eqn:Hopf_Cole_BoLM_Conservative}--\eqref{Eqn:Hopf_Cole_vBC_Conservative} with the prescribed conditions $\left( v_{n}^{(1)}(\mathbf{x}), \widetilde{p}_{\mathrm{p}}^{(1)}(\mathbf{x})\right)$ and $\left( v_{n}^{(2)}(\mathbf{x}), \widetilde{p}_{\mathrm{p}}^{(2)}(\mathbf{x})\right)$, respectively. If 
    %-------------------------------------;
    % Conditions_for_comparison_principle ;
    %-------------------------------------;
    \begin{subequations}
            \begin{alignat}{2}
                \label{Eqn:Hopf_Cole_conditions_for_comparison_principle_velocity}
                &v^{(1)}_{n}(\mathbf{x}) \geq v^{(2)}_{n}(\mathbf{x}) \\
                \label{Eqn:Hopf_Cole_conditions_for_comparison_principle_pressure}
                &\widetilde{p}_{\mathrm{p}}^{(2)}(\mathbf{x}) \geq  \widetilde{p}_{\mathrm{p}}^{(1)}(\mathbf{x})
            \end{alignat}
    \end{subequations}
    then 
    \begin{align}
        \widetilde{p}^{(2)}(\mathbf{x}) \geq \widetilde{p}^{(1)}(\mathbf{x}) \quad \forall \mathbf{x}\in\overline{\Omega}
    \end{align}
\end{theorem}
%---------------------------------------;
%  Proof: Comparison principle (Barus)  ;
%---------------------------------------;
\begin{proof}
    The proof will use Lemma~\ref{Lemma:Hopf_Cole_Comparision_principle_Darcy}, which establishes a comparison principle for the transformed variable: $\mathcal{P}(\mathbf{x})$. Since the transformation under the Hopf-Cole approach (i.e., Eq.~\eqref{Eqn:Hopf_Cole_Functional_form_of_p}) is monotonically increasing, we have: 
    %------------------------------------;
    %  Equation: Ordering of mathcalP_p  ;
    %------------------------------------;
    \begin{align}
        \widetilde{p}_{\mathrm{p}}^{(2)}(\mathbf{x}) 
        \geq  \widetilde{p}_{\mathrm{p}}^{(1)}(\mathbf{x}) 
        \implies 
        \mathcal{P}_{\mathrm{p}}^{(2)}(\mathbf{x}) 
        \geq  \mathcal{P}_{\mathrm{p}}^{(1)}(\mathbf{x})
    \end{align}
    The Hopf-Cole transformation does not alter the velocity field, and hence, Eq.~\eqref{Eqn:Hopf_Cole_conditions_for_comparison_principle_velocity} remains unaltered. Now, invoking Lemma~\ref{Lemma:Hopf_Cole_Comparision_principle_Darcy} on $\mathcal{P}(\mathbf{x})$, we establish: 
    %----------------------------------;
    %  Equation: Ordering of mathcalP  ;
    %----------------------------------;
    \begin{align}
        \mathcal{P}^{(2)}(\mathbf{x}) \geq \mathcal{P}^{(1)}(\mathbf{x}) \quad \forall \mathbf{x}\in\overline{\Omega}
    \end{align} 
    The transformation’s bijectivity and monotonicity-preserving properties yield the desired ordering for $\widetilde{p}(\mathbf{x})$:
    %--------------------------------;
    %  Equation: Ordering of ptilde  ;
    %--------------------------------;
    \begin{align}
        \widetilde{p}^{(2)}(\mathbf{x}) \geq 
        \widetilde{p}^{(1)}(\mathbf{x}) \quad 
        \forall \mathbf{x}\in\overline{\Omega}
    \end{align}    
    %%%
\end{proof}

%====================================;
%  Theorem: Uniqueness of solutions  ;
%====================================;
\begin{theorem}[Uniqueness of solutions under the Barus model]
  \label{Thm:Hopf_Cole_Uniqueness_of_solutions_under_Barus_model}
   The solutions to the boundary value problem \eqref{Eqn:Hopf_Cole_BoLM_Conservative}--\eqref{Eqn:Hopf_Cole_vBC_Conservative} are unique.  
\end{theorem}
%-------------------------------;
% Proof: Uniqueness_Barus_Model ;
%-------------------------------;
\begin{proof}
    This result will be a direct consequence of the comparison principle given by Theorem \ref{Thm:Hopf_Cole_comparison_principle_Barus}, in combination with the method of contradiction and the following property of equalities:
    %------------------------------------;
    %  Equation: Property of equalities  ;
    %------------------------------------;
    \begin{align}
        a = b \iff a \geq b \; \text{and} \; a \leq b 
    \end{align}
    
    In contrast to the hypothesis, assume that there exists two solutions for Eqs.~\eqref{Eqn:Hopf_Cole_BoLM_Conservative}--\eqref{Eqn:Hopf_Cole_vBC_Conservative}:
     
    \[
      \big(\widetilde{p}^{(1)}(\mathbf{x}),\mathbf{v}^{(1)}(\mathbf{x})\big) 
      \quad \mathrm{and} \quad \big(\widetilde{p}^{(2)}(\mathbf{x}),\mathbf{v}^{(2)}(\mathbf{x})\big) 
   \]
    Since they are the solution of the boundary value problems, they must satisfy the given boundary conditions: 
    %---------------------------------------;
    %  Equation: Two solutions satisfy BCs  ;
    %---------------------------------------;
    \begin{align}
        \label{Eqn:Hopf_Cole_uniqueness_Barus_BCs_of_1}
        \widetilde{p}^{(1)}(\mathbf{x}) = \widetilde{p}^{(1)}_{\mathrm{p}}(\mathbf{x}) = \widetilde{p}_{\mathrm{p}}(\mathbf{x})
        \; \mathrm{on} \; \Gamma^{p} 
        \quad \mathrm{and} \quad 
        \mathbf{v}^{(1)}(\mathbf{x}) \bullet \widehat{\mathbf{n}}(\mathbf{x}) 
        = v^{(1)}_{n}(\mathbf{x}) = v_{n}(\mathbf{x})
        \; \mathrm{on} \; \Gamma^{v} \\
        \label{Eqn:Hopf_Cole_uniqueness_Barus_BCs_of_2}
        \widetilde{p}^{(2)}(\mathbf{x}) = \widetilde{p}^{(2)}_{\mathrm{p}}(\mathbf{x}) = \widetilde{p}_{\mathrm{p}}(\mathbf{x})
        \; \mathrm{on} \; \Gamma^{p} 
        \quad \mathrm{and} \quad 
        \mathbf{v}^{(2)}(\mathbf{x}) \bullet \widehat{\mathbf{n}}(\mathbf{x}) 
        = v^{(2)}_{n}(\mathbf{x}) = v_{n}(\mathbf{x})
        \; \mathrm{on} \; \Gamma^{v}
    \end{align}

    In view of the equalities in Eqs.~\eqref{Eqn:Hopf_Cole_uniqueness_Barus_BCs_of_1} and \eqref{Eqn:Hopf_Cole_uniqueness_Barus_BCs_of_2}, we consider the following inequalities: 
    %------------------------------------------;
    %  Equation: One side inequalities of BCs  ;
    %------------------------------------------;
     \begin{align}
        \label{Eqn:Hopf_Cole_uniqueness_Barus_p1p_geq_p2p}
        \widetilde{p}^{(1)}_{\mathrm{p}}(\mathbf{x})
        \geq \widetilde{p}^{(2)}_{\mathrm{p}}(\mathbf{x}) 
        \; \mathrm{on} \, \Gamma^{p} 
        \quad \mathrm{and} \quad 
        v_{n}^{(1)}(\mathbf{x}) \leq  v_{n}^{(2)}(\mathbf{x})
        \; \mathrm{on} \, \Gamma^{v} 
    \end{align}
    Then, Theorem \ref{Thm:Hopf_Cole_comparison_principle_Barus} implies: 
    %-----------------------;
    %  Equation: p1 geq p2  ;
    %-----------------------;
    \begin{align}
       \label{Eqn:Hopf_Cole_uniqueness_Barus_p1_geq_p2}
       \widetilde{p}^{(1)}(\mathbf{x}) \geq   \widetilde{p}^{(2)}(\mathbf{x}) \quad \forall \mathbf{x} \in \overline{\Omega}
   \end{align}
   Now proceeding in the reverse direction, we consider the following inequalities: 
   %--------------------------------------------;
   %  Equation: Other side inequalities of BCs  ;
   %--------------------------------------------;
     \begin{align}
        \label{Eqn:Hopf_Cole_uniqueness_Barus_p2p_geq_p1p}
        \widetilde{p}^{(2)}_{\mathrm{p}}(\mathbf{x})
        \geq \widetilde{p}^{(1)}_{\mathrm{p}}(\mathbf{x}) 
        \; \mathrm{on} \, \Gamma^{p} 
        \quad \mathrm{and} \quad 
        v_{n}^{(2)}(\mathbf{x}) \leq  v_{n}^{(1)}(\mathbf{x})
        \; \mathrm{on} \, \Gamma^{v} 
    \end{align}
    Again, Theorem \ref{Thm:Hopf_Cole_comparison_principle_Barus} implies: 
    %-----------------------;
    %  Equation: p2 geq p1  ;
    %-----------------------;
    \begin{align}
        \label{Eqn:Hopf_Cole_uniqueness_Barus_p2_geq_p1}
       \widetilde{p}^{(2)}(\mathbf{x}) \geq   \widetilde{p}^{(1)}(\mathbf{x}) \quad \forall \mathbf{x} \in \overline{\Omega}
   \end{align}
   Inequalities \eqref{Eqn:Hopf_Cole_uniqueness_Barus_p1_geq_p2} and \eqref{Eqn:Hopf_Cole_uniqueness_Barus_p2_geq_p1} collectively establish: 
   %-----------------------------------;
   %  Equation: Equality of pressures  ;
   %-----------------------------------;
   \begin{align}
       \widetilde{p}^{(1)}(\mathbf{x})  =  \widetilde{p}^{(2)}(\mathbf{x}) 
       \quad \forall \mathbf{x} \in \overline{\Omega}
   \end{align}
   Finally, using the balance of linear momentum \eqref{Eqn:Hopf_Cole_BoLM_Conservative} we assert that 
   %------------------------------------;
   %  Equation: Equality of velocities  ;
   %------------------------------------;
   \begin{align}
       \mathbf{v}^{(1)}(\mathbf{x}) 
       = \mathbf{v}^{(2)}(\mathbf{x}) 
       \quad \forall \mathbf{x} \in \overline{\Omega}
   \end{align}
   %%%
\end{proof}

%--------------------------------------------------------;
%  Remark: Remark about alternative proof to uniqueness  ;
%--------------------------------------------------------;
\begin{remark}
    An alternative proof of Theorem \ref{Thm:Hopf_Cole_Uniqueness_of_solutions_under_Barus_model} can be constructed by utilizing the uniqueness of solutions for the transformed variable $\mathcal{P}(\mathbf{x})$, as given by Lemma \ref{Lemma:Hopf_Cole_Uniqueness_Darcy}, and the bijective property of the transformation \eqref{Eqn:Hopf_Cole_Functional_form_of_p}. We do not present this alternative proof here for the sake of brevity.
\end{remark}

%=============================;
%  Subsection: Ill-posedness  ;
%=============================;
\subsection{Ill-posedness of some velocity-driven problems under the Barus model}
\label{Sec:Cole_Hopf_Ill_posedness}

We now show that a class of boundary value problems with mixed velocity-pressure boundary conditions has no solutions under the Barus model. We refer to this class as velocity-driven problems.

In its simplest form, consider a one-dimensional problem as shown in Fig.~\ref{Fig:Hopf_Cole_Existence_of_soluitions}. The domain has length $L$, with a prescribed velocity $v_0$ at the left end and a prescribed pressure $p = p_0$ at the right end. Additionally, we assume that the body force $\rho \mathbf{b}$ is zero and that the permeability is constant.

%----------------------------------;
%  Figure 3: 1D ill-posed problem  ;
%----------------------------------;
\begin{figure}[ht]
    \centering
    \includegraphics[width=0.55\linewidth]{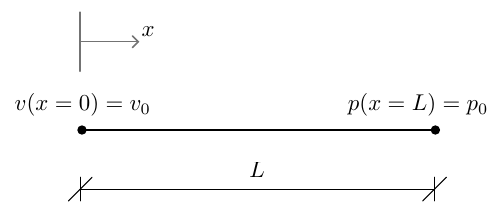}
    \caption{An ill-posed velocity-driven boundary value problem under the Barus model.}
    \label{Fig:Hopf_Cole_Existence_of_soluitions}
\end{figure}

%=================================================;
%  Subsubsection: Solution under the Barus model  ;
%=================================================;
\subsubsection{Solution directly using the Barus model}
The one-dimensional boundary value problem reads as follows: 
%--------------------;
%  Equation: 1D BVP  ;
%--------------------;
\begin{subequations}
    \begin{alignat}{2}
        \label{Eqn:Hopf_Cole_1D_BoLM}
        &\frac{\mu_0}{k} \, \exp\left[\beta \, \left(\frac{p(x)}{p_0} - 1\right)\right] \, v(x) + \frac{\mathrm{d} \, p(x)}{\mathrm{d} \, x} = 0 
        &&\quad \forall x \in (0,L) \\
        \label{Eqn:Hopf_Cole_1D_BoM}
        &\frac{\mathrm{d} \, v(x)}{\mathrm{d} \, x} = 0 
        &&\quad \forall x \in (0,L) \\
        \label{Eqn:Hopf_Cole_1D_pBC}
        &p(x) = p_R = p_0 
        &&\quad \mathrm{at} \;  x = L \\
        \label{Eqn:Hopf_Cole_1D_vBC}
        &v(x) = v_0 
        &&\quad \mathrm{at} \;  x = 0        
    \end{alignat}
\end{subequations}

Equations \eqref{Eqn:Hopf_Cole_1D_BoM} and \eqref{Eqn:Hopf_Cole_1D_vBC} imply that 
%-------------------------;
%  Equation: Solving BoM  ;
%-------------------------;
\begin{align}
    \label{Eqn:Hopf_Cole_1D_Barus_v}
    v(x) = v_0 \quad \forall x \in [0,L]
\end{align}
Using the solution for the velocity field, Eq.~\eqref{Eqn:Hopf_Cole_1D_BoLM} can be rewritten as follows: 
%----------------------------;
%  Equation: Rewriting BoLM  ;
%----------------------------;
\begin{align}
    \label{Eqn:Hopf_Cole_1D_GE}
    \frac{\mathrm{d} \, p}{\mathrm{d} \, x} 
    = - \frac{\mu_{0} \, v_0}{k} \, \text{exp}\left[\beta\left(\frac{p}{p_{0}}- 1\right)\right] 
\end{align}
Integrating over the spatial interval $(x, L)$, we write:
%-----------------------------------------;
%  Equation: Integrating over the limits  ;
%-----------------------------------------;
\begin{align}
    \int^{p_R = p_0}_{p(x)} \text{exp}\left[-\beta\left(\frac{p}{p_{0}}- 1\right)\right] \mathrm{d}p = \int^{L}_{x} -\frac{\mu_0 \, v_{0}}{k} \, \mathrm{d}x
\end{align}
Upon executing the integrals, we get: 
%-------------------------------------;
%  Equation: Executing the integrals  ;
%-------------------------------------;
\begin{align}
    \text{exp}\left[-\beta\left(\frac{p(x)}{p_{0}}- 1\right)\right] = 1 - \frac{\mu_{0} \, v_{0} \, \beta}{p_{0} \, k} \, (L - x)
\end{align}
Rearranging the terms, we get the following expression for the pressure: 
%----------------------------;
%  Equation: Solution for p  ;
%----------------------------;
\begin{align}
    \label{Eqn:Hopf_Cole_1D_Barus_p}
    p(x) = p_0 \, \left( 1 -\frac{1}{\beta} \, \ln\left[1 - \frac{\mu_{0} \, v_{0} \, \beta}{p_{0} \, k} \, (L - x)\right] \right) 
\end{align}
Equations \eqref{Eqn:Hopf_Cole_1D_Barus_v} and \eqref{Eqn:Hopf_Cole_1D_Barus_p} are the analytical expressions for the solutions fields, obtained by solving the Barus model directly. 

However, the solution is not defined for any prescribed inputs. If the argument of the logarithmic function in Eq.~\eqref{Eqn:Hopf_Cole_1D_Barus_p} is negative, then the solution for the pressure does not exist. Thus, a solution exists if and only if the following condition holds:
%-------------------------------;
%  Equation: Spatial condition  ;
%-------------------------------;
\begin{align}
    L - x < \frac{p_{0} \, k}{\mu_{0} \, v_{0} \, \beta}  \quad \forall x \in [0,L]
\end{align}
The worst-case scenario occurs when $x=0$. Substituting $x=0$ into the preceding inequality, we obtain the condition that must be satisfied for solutions to exist: 
%------------------------------;
%  Equation: Condition is met  ;
%------------------------------;
\begin{align}
    \label{Eqn:Hopf_Cole_Existence_condition}
    v_{0} < \frac{p_0 \, k}{\mu_0 \, L \, \beta}
\end{align}

%==============================================================;
%  Subsubsection: Solution under the Hopf-Cole transformation  ;
%==============================================================;
\subsubsection{Solution using the Hopf-Cole transformation}
The transformed one-dimensional boundary value problem (cf. Eqs.~ \eqref{Eqn:Hopf_Cole_Linear_BoLM}--\eqref{Eqn:Hopf_Cole_Linear_vBC}) takes the following form: 
%--------------------------------;
%  Equation: 1D transformed BVP  ;
%--------------------------------;
\begin{subequations}
    \begin{alignat}{2}
        \label{Eqn:Hopf_Cole_1D_Linear_BoLM}
        &\frac{\mu_0}{k} \, v(x) + \frac{\mathrm{d} \, \mathcal{P}(x)}{\mathrm{d} \, x} = 0 
        &&\quad \forall x \in (0,L) \\
        \label{Eqn:Hopf_Cole_1D_Linear_BoM}
        &\frac{\mathrm{d} \, v(x)}{\mathrm{d} \, x} = 0 
        &&\quad \forall x \in (0,L) \\
        \label{Eqn:Hopf_Cole_1D_Linear_pBC}
        &\mathcal{P}(x) = -\frac{p_0}{\beta} 
        &&\quad \mathrm{at} \;  x = L \\
        \label{Eqn:Hopf_Cole_1D_Linear_vBC}
        &v(x) = v_0 
        &&\quad \mathrm{at} \;  x = 0        
    \end{alignat}
\end{subequations}

Similar to the derivation presented in the preceding section, Eqs.~\eqref{Eqn:Hopf_Cole_1D_Linear_BoM} and \eqref{Eqn:Hopf_Cole_1D_Linear_vBC} imply: 
%-------------------------------;
%  Equation: Velocity solution  ;
%-------------------------------;
\begin{align}
    \label{Eqn:Hopf_Cole_1D_transformed_v}
    v(x) = v_0 \quad \forall x \in [0,L]
\end{align}
Consequently, Eq.~\eqref{Eqn:Hopf_Cole_1D_Linear_BoLM} becomes: 
%-------------------------------;
%  Equation: Intermediate step  ;
%-------------------------------;
\begin{align}
    \frac{\mu_0 \, v_0}{k} 
    + \frac{\mathrm{d} \, \mathcal{P}(x)}{\mathrm{d} \, x} = 0 
\end{align}
Integrating over the spatial interval $(x, L)$, we write at: 
%-----------------------------------;
%  Equation: Integrate from x to L  ;
%-----------------------------------;
\begin{align}
    \mathcal{P}(x = L) - \mathcal{P}(x) 
    = -\frac{\mu_0 \, v_0}{k} \, (L -x)
\end{align}
Enforcing the pressure boundary condition \eqref{Eqn:Hopf_Cole_1D_Linear_pBC}, we arrive:
%------------------------------------;
%  Equation: Solution for mathcal_P  ;
%------------------------------------;
\begin{align}
    \label{Eqn:Hopf_Cole_1D_transformed_p}
    \mathcal{P}(x) = -\frac{p_0}{\beta} + \frac{\mu_0 \, v_0}{k} \, (L-x)
\end{align}
Equations \eqref{Eqn:Hopf_Cole_1D_transformed_v} and \eqref{Eqn:Hopf_Cole_1D_transformed_p} provide analytical expressions for the transformed pressure $\mathcal{P}(x)$ and $v(x)$.

We use the transformation \eqref{Eqn:Hopf_Cole_Functional_form_of_original_p} to find $p(x)$ from the obtained transformed pressure $\mathcal{P}(x)$. Note that $\xi(x) = 0$ as the body force is absent in the boundary value problem. But for the transformation to be well-defined, $\mathcal{P}(x)$ should be negative everywhere in the domain---otherwise, the argument in the logarithmic function will be negative, and the expression in Eq.~\eqref{Eqn:Hopf_Cole_Functional_form_of_original_p} is not well-defined. 

By virtue of Eq.~\eqref{Eqn:Hopf_Cole_1D_transformed_p}, $\mathcal{P}(x) < 0$ leads to  
%------------------------------;
%  Equation: Basic inequality  ;
%------------------------------;
\begin{align}
    \frac{p_0}{\beta} < \frac{\mu_0 \, v_0}{k} \, (L-x)
\end{align}
The preceding inequality must hold for all $x \in [0,L]$. Evaluating for the most critical case (i.e., $x = 0$), we get the following condition, which must be met for the transformation to be well-defined: 
%------------------------------;
%  Equation: Final inequality  ;
%------------------------------;
\begin{align}
     v_{0} < \frac{p_0 \, k}{\mu_0 \, L \, \beta}
\end{align}
This condition matches the result obtained earlier by directly working with the Barus model (cf. Eq.~\eqref{Eqn:Hopf_Cole_Existence_condition}). 

The discussion so far highlights two main points: 
%--------------------------------------;
%  Enumerate: Highlights of existence  ;
%--------------------------------------;
\begin{enumerate}
    \item[{[A]}] Solutions might not exist for a class of velocity-driven problems under the Barus model, as illustrated in the preceding example. 
    \item[{[B]}] The Hopf-Cole transformation also predicts the non-existence of solutions for this class of velocity-driven problems. Notably, the transformation converts the boundary value problem under the Barus model into the Darcy equations, for which solutions do exist. The non-existence of a solution arises from the failure of the inversion (going from $\mathcal{P}(\mathbf{x})$ to $\widetilde{p}(\mathbf{x})$), as $\mathcal{P}(\mathbf{x})$, the solution to the transformed boundary value problem, does not necessarily lie within $(-\infty, 0)$. According to Eq.~\eqref{Eqn:Hopf_Cole_Functional_form_of_mathcal_P}, a positive value of $\mathcal{P}(\mathbf{x})$ would make the argument of the logarithm negative, implying that no real number $\widetilde{p}(\mathbf{x})$ can satisfy the equation.
\end{enumerate}

    %*********************************************;
%                                             ;
%  NAME                                       ;
%    S5_Hopf_Cole_Mechanics.tex               ;
%                                             ;
%*********************************************;
\section{RECIPROCITY: A MECHANICS-BASED PRINCIPLE}
\label{Sec:S5_Hopf_Cole_Mechanics}
In the field of mechanics, various classes of problems---such as elastostatics and elastodynamics---have been shown to exhibit reciprocity \citep{gurtin1973linear,achenbach2003reciprocity}. 
In the context of porous media, reciprocity has also been demonstrated for the Darcy and Darcy-Brinkman models \citep{shabouei2016mechanics}---also see Appendix \ref{App:Hopf_Cole_Reciprocal}, and more recently for the double porosity/permeability model \citep{nakshatrala2018modeling,nakshatrala2024dynamic}. Importantly, all these aforementioned problems are inherently linear. In what follows, we establish reciprocity for the \emph{nonlinear} Barus model, made possible through the proposed approach outlined in Fig.~\ref{Fig:Hopf_Cole_Flowchart}.

%======================================;
%  Theorem: Reciprocal relation Barus  ;
%--------------------------------------;
\begin{theorem}[A reciprocal relation on $\widetilde{p}(\mathbf{x})$]
    \label{Thm:Hopf_Cole_Betti_Theorem_Barus}
    Let $\big(\widetilde{p}^{(1)}(\mathbf{x}), \mathbf{v}^{(1)}(\mathbf{x})\big)$ and $\big(\widetilde{p}^{(2)}(\mathbf{x}),\mathbf{v}^{(2)}(\mathbf{x})\big)$ be the solutions of Eqs.~\eqref{Eqn:Hopf_Cole_BoLM_Conservative}--\eqref{Eqn:Hopf_Cole_BoM_Conservative} corresponding to the prescribed boundary conditions $\big(\widetilde{p}_{\mathrm{p}}^{(1)}(\mathbf{x}),v^{(1)}_n(\mathbf{x})\big)$ and $\big(\widetilde{p}_{\mathrm{p}}^{(2)}(\mathbf{x}),v^{(2)}_n(\mathbf{x})\big)$, respectively. These boundary conditions and the associated solution fields satisfy the following reciprocal relation:
    %---------------------------------------;
    %  Equation: Barus reciprocal relation  ;
    %---------------------------------------;
    \begin{align}
        \label{Eqn:Hopf_Cole_Betti_relation_Barus}
        %%%
        &\int_{\Gamma^{v}} v_{n}^{(2)}(\mathbf{x}) 
        \, \text{exp}\left[ -\beta \left(\frac{\widetilde{p}^{(1)}(\mathbf{x})}{p_0} - 1 \right) \right] \, \mathrm{d} \Gamma 
        - \int_{\Gamma^{p}}
        \text{exp}\left[ -\beta \left(\frac{\widetilde{p}_{\mathrm{p}}^{(2)}(\mathbf{x})}{p_0} - 1 \right) \right] \, \mathbf{v}^{(1)}(\mathbf{x}) 
        \bullet \widehat{\mathbf{n}}(\mathbf{x}) 
        \, \mathrm{d} \Gamma \nonumber \\
        %%%
        &\hspace{1.5in} = \int_{\Gamma^{v}}
        v^{(1)}_{n}(\mathbf{x}) \,\text{exp}\left[ -\beta \left(\frac{\widetilde{p}^{(2)}(\mathbf{x})}{p_0} - 1 \right) \right]
        \, \mathrm{d} \Gamma  \nonumber \\ 
        &\hspace{2in}- \int_{\Gamma^{p}}
        \text{exp}\left[ -\beta \left(\frac{\widetilde{p}_{\mathrm{p}}^{(1)}(\mathbf{x})}{p_0} - 1 \right) \right]
        \, \mathbf{v}^{(2)}(\mathbf{x}) 
        \bullet \widehat{\mathbf{n}}(\mathbf{x}) 
        \, \mathrm{d} \Gamma 
        %%%
    \end{align}  
    %--------------------------------;
    %  Proof of reciprocal relation  ;
    %--------------------------------;
    \begin{proof}
        Our strategy is to follow the steps in the proposed approach, as outlined in the flowchart (Fig.~\ref{Fig:Hopf_Cole_Flowchart}), in conjunction with Theorem \ref{Thm:Hopf_Cole_Betti_Theorem_Darcy}. Executing Step 1, we get $\mathcal{P}_{\mathrm{p}}^{(1)}(\mathbf{x})$ and $\mathcal{P}_{\mathrm{p}}^{(2)}(\mathbf{x})$ from $\widetilde{p}^{(1)}_{\mathrm{p}}(\mathbf{x})$ and $\widetilde{p}^{(2)}_{\mathrm{p}}(\mathbf{x})$, respectively. Velocity boundary conditions---that is, $v_n^{(1)}(\mathbf{x})$ and $v_n^{(2)}(\mathbf{x})$---are unaltered. 
        
        As per Step 2, we get $\big(\mathcal{P}^{(1)}(\mathbf{x}),\mathbf{v}^{(1)}(\mathbf{x})\big)$ and $\big(\mathcal{P}^{(2)}(\mathbf{x}),\mathbf{v}^{(2)}(\mathbf{x})\big)$ by solving the boundary value problem \eqref{Eqn:Hopf_Cole_Linear_BoLM}--\eqref{Eqn:Hopf_Cole_Linear_vBC} under the prescribed conditions $\big(\mathcal{P}_{\mathrm{p}}^{(1)}(\mathbf{x}),v_n^{(1)}(\mathbf{x})\big)$ and $\big(\mathcal{P}_{\mathrm{p}}^{(2)}(\mathbf{x}),v^{(2)}_n(\mathbf{x})\big)$, respectively. Next, from Theorem \ref{Thm:Hopf_Cole_Betti_Theorem_Darcy}, presented in Appendix \ref{App:Hopf_Cole_Reciprocal}, we have the following relation: 
        %-------------------------------------------;
        %  Equation: Reciprocal relation for Darcy  ;
        %-------------------------------------------;
        \begin{align}
            \int_{\Gamma^{v}} v_{n}^{(2)}(\mathbf{x}) 
            \, \mathcal{P}^{(1)}(\mathbf{x}) \, \mathrm{d} \Gamma 
            - \int_{\Gamma^{p}}
            \mathcal{P}^{(2)}_{\mathrm{p}}(\mathbf{x}) 
            \, \mathbf{v}^{(1)}(\mathbf{x}) 
            \bullet \widehat{\mathbf{n}}(\mathbf{x}) 
            \, \mathrm{d} \Gamma 
            &=
            \int_{\Gamma^{v}}
            v^{(1)}_{n}(\mathbf{x}) \, \mathcal{P}^{(2)}(\mathbf{x}) 
            \, \mathrm{d} \Gamma \nonumber \\ 
            &\hspace{0.4in} - \int_{\Gamma^{p}}
            \mathcal{P}^{(1)}_{\mathrm{p}}(\mathbf{x}) 
            \, \mathbf{v}^{(2)}(\mathbf{x}) 
            \bullet \widehat{\mathbf{n}}(\mathbf{x}) 
            \, \mathrm{d} \Gamma 
        \end{align}
        Now, executing Step 3---invoking the transformation \eqref{Eqn:Hopf_Cole_Functional_form_of_mathcal_P} on $\mathcal{P}^{(1)}(\mathbf{x})$, $\mathcal{P}^{(2)}(\mathbf{x})$, $\mathcal{P}^{(1)}_{\mathrm{p}}(\mathbf{x})$, and $\mathcal{P}_{\mathrm{p}}^{(2)}(\mathbf{x})$, we obtain: 
        \begin{align}
            &\int_{\Gamma^{v}} v_{n}^{(2)}(\mathbf{x}) 
            \, p_0 \, \text{exp}\left[-\beta 
            \left(\frac{\widetilde{p}^{(1)}(\mathbf{x})}{p_0} - 1 \right) 
            \right] \, \mathrm{d} \Gamma 
            - \int_{\Gamma^{p}}
            p_0 \, \text{exp}\left[ -\beta \left(\frac{\widetilde{p}_{\mathrm{p}}^{(2)}(\mathbf{x})}{p_0} - 1 \right) \right] 
            \, \mathbf{v}^{(1)}(\mathbf{x}) \bullet 
            \widehat{\mathbf{n}}(\mathbf{x}) 
            \, \mathrm{d} \Gamma \nonumber \\
            &\hspace{1.5in} = \int_{\Gamma^{v}}
            v^{(1)}_{n}(\mathbf{x}) \, p_0 \,\text{exp}\left[ -\beta \left(\frac{\widetilde{p}^{(2)}(\mathbf{x})}{p_0} - 1 
            \right) \right] \, \mathrm{d} \Gamma  \nonumber \\ 
            &\hspace{2.5in} - \int_{\Gamma^{p}}
            p_0 \, \text{exp}\left[ -\beta \left(\frac{\widetilde{p}_{\mathrm{p}}^{(1)}(\mathbf{x})}{p_0} - 1 \right) \right]
            \, \mathbf{v}^{(2)}(\mathbf{x}) 
            \bullet \widehat{\mathbf{n}}(\mathbf{x}) 
            \, \mathrm{d} \Gamma 
           %%%
        \end{align}  
     Finally, canceling the factor $p_0$ from both sides of the above equation---as it is a constant---yields the desired reciprocal relation given by Eq.~\eqref{Eqn:Hopf_Cole_Betti_relation_Barus}.
    \end{proof}
\end{theorem}

To the best of our knowledge, this is the first instance where reciprocity has been demonstrated for a nonlinear problem. A cursory review of the literature might suggest otherwise: \textsc{Clifford Truesdell} addressed reciprocity in the context of hyperelasticity---a nonlinear phenomenon in solid mechanics \citep{truesdell1963meaning}. However, a more careful examination reveals that his work actually dealt with superposed infinitesimal strains, which are linear in terms of displacements. Moreover, his focus was on establishing a connection between reciprocity and the existence of a potential, which, in the context of hyperelasticity, corresponds to the stored strain energy functional.

Finally, the above reciprocal relation can be used as an \emph{a priori} estimator to assess the accuracy of a numerical formulation. For example, numerical solutions can be obtained for a given problem under two prescribed conditions. Using the prescribed quantities and the obtained numerical solution fields, one can then check how well Eq.~\eqref{Eqn:Hopf_Cole_Betti_relation_Barus} is satisfied.

    %*********************************************;
%                                             ;
%  NAME                                       ;
%    S6_Hopf_Cole_NR.tex                      ;
%                                             ;
%*********************************************;
\section{INSIGHTS USING A CANONICAL EXAMPLE}
\label{Sec:S6_Hopf_Cole_NR}
In this section, we establish the accuracy, computational, and analytical advantages of the proposed approach, which avails the Hopf-Cole transformation. To this end, we consider a canonical problem---the reservoir problem, which is widely used in the literature; for example, see \citep{nakshatrala2011numerical,chang2017modification}. 

Consider a reservoir, idealized with a unit width along the out-of-the-plane direction and a rectangular cross-section with a length of $L = 100\, \mathrm{m}$ and a height of $H = 30 \, \mathrm{m}$, as shown in Fig.~\ref{Fig:Hopf_Cole_Reservoir_BVP}. The length of the production well is $W = 0.2 \, \mathrm{m}$. The left side is subjected to an injection pressure $p_{\mathrm{inj}}$, while the pressure at the production well is equal to the atmospheric pressure $p_{\mathrm{atm}}$. On the rest of the faces, the normal component of the velocity is zero (i.e., no penetration boundary condition). Table \ref{Table:Hopf_Cole_Parameters} provides the simulation parameters.

The boundary value problem was solved using the FEniCS framework---an open-source library for solving partial differential equations via the finite element method \citep{logg2012automated}.
The mixed Galerkin formulation was employed \citep{brezzi2012mixed}, and three-node elements were used to discretize the computational domain, with the number of nodes given by $n_{\mathrm{nodes}} = 4042$ and the number of elements by $n_{\mathrm{elements}} = 8082$. To satisfy the Ladyzhenskaya-Babu\v{s}ka-Brezzi (LBB) stability condition, Taylor–Hood spaces were used, with second-order interpolation for the velocity and first-order interpolation for the pressure \citep{ern2004theory}.

%---------------------------------------;
%  Figure 4: Reservoir BVP description  ;
%---------------------------------------;
\begin{figure}[h!]
    \centering
    \includegraphics[width=0.65\linewidth]{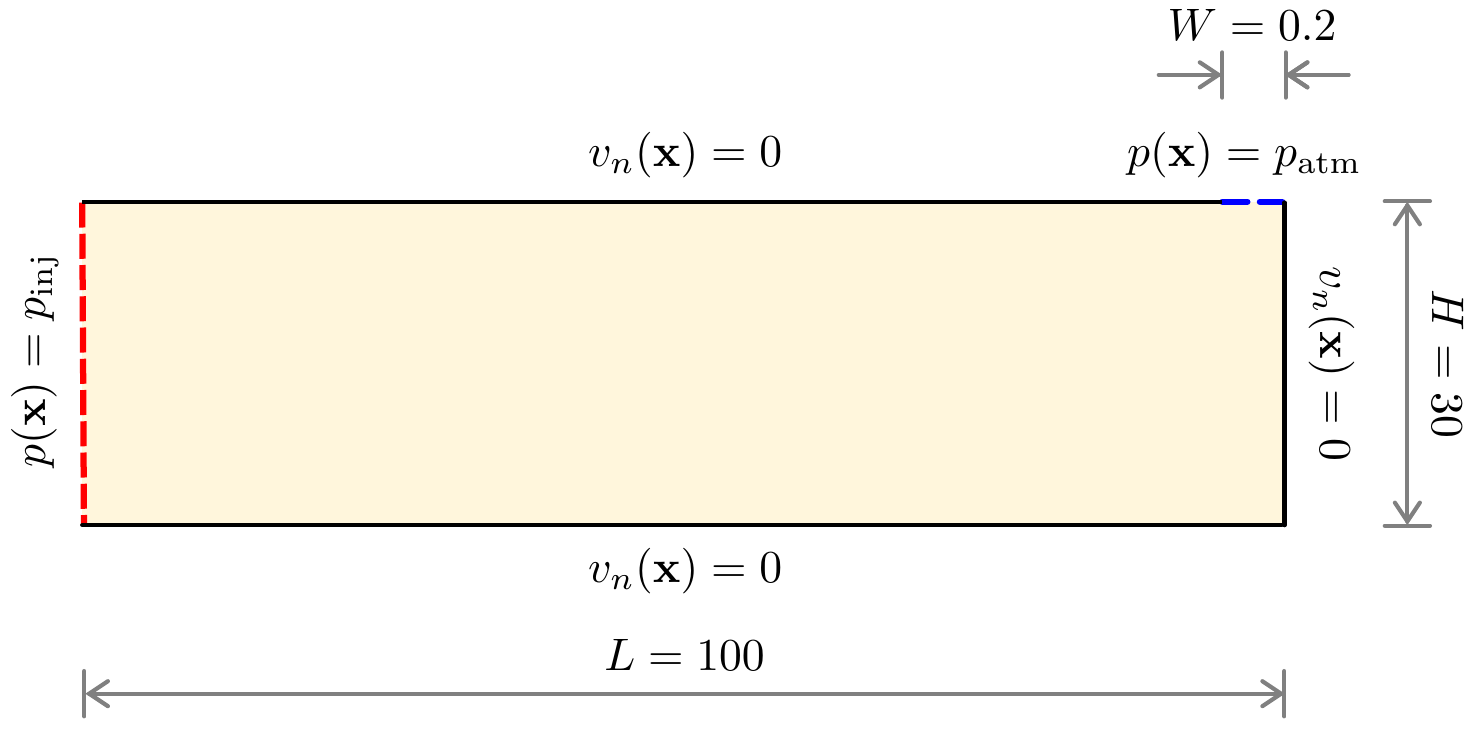}
    \caption{Problem description of the reservoir problem. The figure shows the 2D cross-section of the reservoir, with a unit width along the out-of-plane direction, which is not depicted in the figure.}
    \label{Fig:Hopf_Cole_Reservoir_BVP}
\end{figure}

%--------------------------------;
%  Table: Simulation parameters  ;
%--------------------------------;
\begin{table}[h!]
    \caption{Parameters used in the numerical simulations. \label{Table:Hopf_Cole_Parameters}}
    \centering
    \begin{tabular}{|c|c|}
        \hline
        \textbf{Parameter} & \textbf{Value} \\
        \hline
        $p_{\text{inj}}$ & $(10-60,000) \, p_{\text{atm}}$ \\
        \hline
        $p_0 = p_{\text{atm}}$ & 101325 $\text{Pa}$ \\
        \hline
        $\beta$ & $3 \times 10^{-6}$\\
        \hline
        $\mu_{0}$ & $3.95 \times 10^{-5} \; \text{Pa}\cdot\text{s}$\\
        \hline
        $k$ & $10^{-12} \; \text{m}^{-2}$\\
        \hline
        $\mathbf{b}(\mathbf{x})$ & $\mathbf{0} = (0,0) \; \text{m}/\text{s}^{2}$\\
        \hline
    \end{tabular}
\end{table}

%======================================;
%  Subsection: Numerical verification  ;
%======================================;
\subsection{Numerical verification and computational advantages}
Figure~\ref{Fig:Hopf_Cole_Solution_profiles} demonstrates that the solution fields under the original and proposed approaches exhibit excellent agreement, validating the accuracy of the proposed approach. However, as shown in Fig.~\ref{Fig:Hopf_Cole_Time_to_solution}, the proposed approach requires only a fraction of the time compared to the original approach. Notably, the proposed approach utilizes a solver for the classical Darcy equations, avoiding the need for a nonlinear solver.

%--------------------------------------------;
%  Figure 5: Pressure and velocity profiles  ;
%--------------------------------------------;
\begin{figure}[h!]
    \centering
    \includegraphics[width=0.8\linewidth]{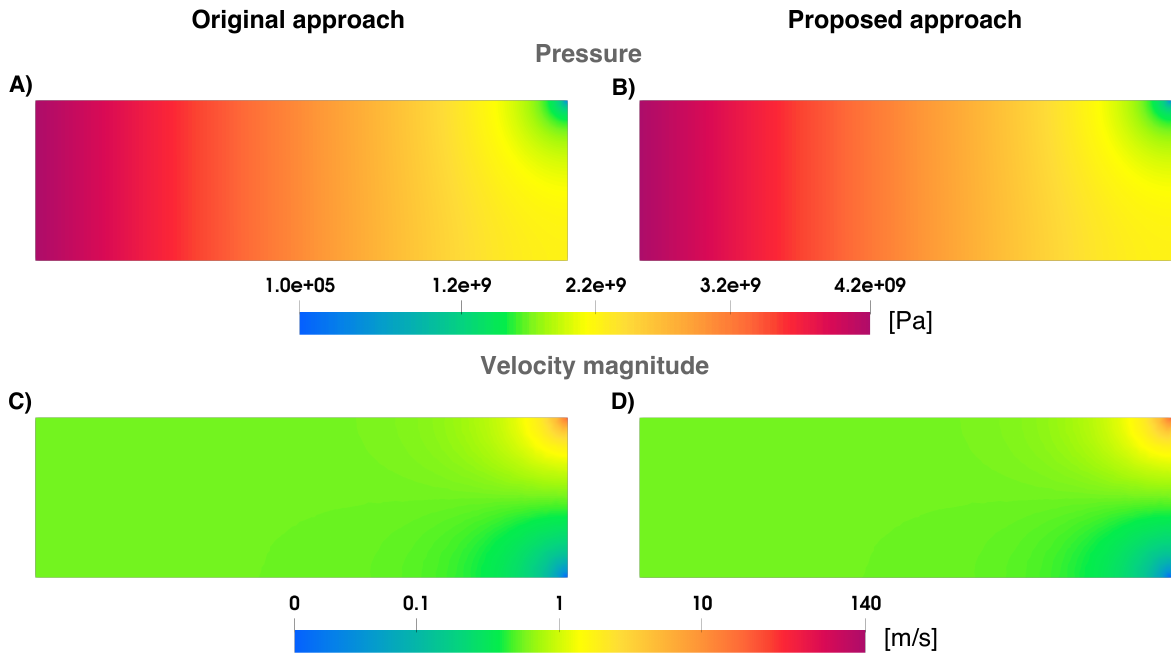}
    \caption{This figure shows that the solution fields using the original (left) and proposed (right) approaches match. Table \ref{Table:Hopf_Cole_Parameters} provides the parameters, except that $p_{\text{inj}} = 4.2\, \times \, 10^{9}$ Pa was used in this study. Note: For enhanced visualization, a logarithmic scale was applied when mapping data to colors while plotting the velocity magnitude profiles (i.e., subfigures C and D)—a feature available in the visualization software, ParaView \citep{paraview}.} \label{Fig:Hopf_Cole_Solution_profiles}
\end{figure}

%------------------------------;
%  Figure 6: Time-to-solution  ;
%------------------------------;
\begin{figure}[h!]
    \centering
    \includegraphics[width=0.6\linewidth]{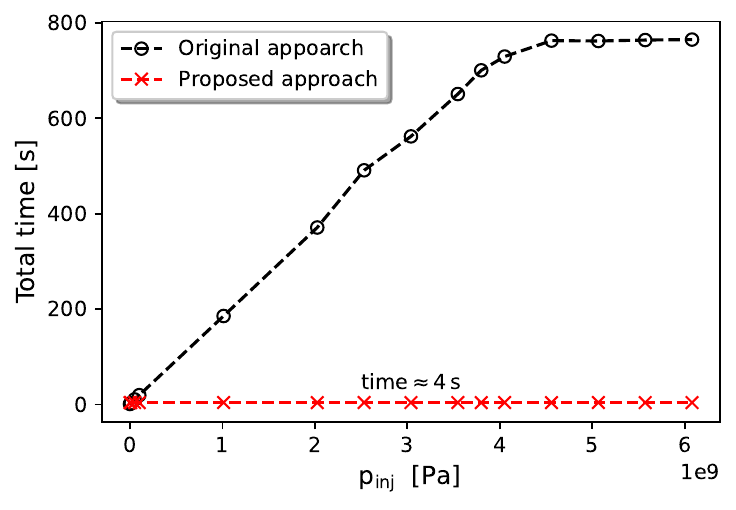}
    \caption{This figure compares the time-to-solution under the original approach (i.e., solving the Barus model directly) versus the proposed approach (i.e., utilizing the Hopf-Cole transformation). \emph{Inference:} The simulation time with the proposed approach is a fraction of that required by the original approach. \label{Fig:Hopf_Cole_Time_to_solution}}
\end{figure}

%============================;
%  Subsection: Ceiling flux  ;
%============================;
\subsection{Ceiling flux: An analytical derivation}
To demonstrate the analytical advantages of the proposed approach, we establish that the Barus model predicts a ceiling flux for the reservoir problem. By ceiling flux, we mean that the volumetric flux approaches an asymptote as the injection pressure increases---a phenomenon that has been observed in practice. In contrast, the classical Darcy model predicts an unphysical trend: a continuous linear increase in flux with increasing driving pressure. \citet{nakshatrala2011numerical} demonstrated the existence of ceiling flux under the Barus model through numerical simulations. Herein, we establish this result analytically by leveraging the proposed approach; see the flowchart provided in Fig.~\ref{Fig:Hopf_Cole_Flowchart}.

Before we pursue the steps under the proposed approach, we note that the problem is driven by the difference between $p_{\mathrm{inj}}$ and $p_{\mathrm{atm}}$. Now, following Step 1, we get $\mathcal{P}_{\mathrm{inj}}$ and $\mathcal{P}_{\mathrm{atm}}$ corresponding to $p_{\mathrm{inj}}$ and $p_{\mathrm{atm}}$. For convenience and for the reason that the flow is driven by the pressure difference, we define 
%----------------------------;
%  Equation: Delta mathcalP  ;
%----------------------------;
\begin{equation}
    \label{Eqn:Hopf_Cole_Delta_mathcalP}
    \Delta \mathcal{P} = \mathcal{P}_{\text{inj}}
    - \mathcal{P}_{\text{atm}}
\end{equation}

Following Step 2, we obtain the solution fields that satisfy the boundary value problem \eqref{Eqn:Hopf_Cole_Linear_BoLM}--\eqref{Eqn:Hopf_Cole_Linear_vBC}. Noting the linearity of the mentioned boundary value problem, we write the transformed pressure field as follows: 
%--------------------------------------;
%  Equation: mathcalP in terms of phi  ;
%--------------------------------------;
\begin{equation}
   \mathcal{P}(\mathbf{x}) = \mathcal{P}_{\text{atm}}\, +\,  \Delta \mathcal{P} \,\mathcal{G}(\mathbf{x})
\end{equation}
where $\mathcal{G}(\mathbf{x})$ satisfy the following auxiliary boundary value problem: 
%---------------------------;
%  Equation: Auxiliary BVP  ;
%---------------------------;
\begin{alignat}{2}
    -&\text{div} \Bigl[\widetilde{\mu}_0^{-1}(\mathbf{x}) \, \mathbf{K}(\mathbf{x}) \, \mathrm{grad}\big[\mathcal{G}(\mathbf{x})\big]\Bigr] = 0 
    &&\quad \mathrm{in} \; \Omega \\
    -&\widehat{\mathbf{n}}(\mathbf{x}) \bullet \widetilde{\mu}_0^{-1}(\mathbf{x}) \, \mathbf{K}(\mathbf{x}) \, \mathrm{grad}\big[\mathcal{G}(\mathbf{x})\big] =  0 
    &&\quad \mathrm{on} \; \Gamma^{v} \\
    &\mathcal{G}(\mathbf{x}) = 1 &&\quad \mathrm{on} \; 
    \Gamma_{\mathrm{inj}} \\
    &\mathcal{G}(\mathbf{x}) = 0 &&\quad \mathrm{on} \; 
    \Gamma_{\mathrm{prod}} 
\end{alignat}
Clearly, $\mathcal{G}(\mathbf{x})$ is entirely determined by the geometry of the domain, permeability $\mathbf{K}(\mathbf{x})$, and $\widetilde{\mu}_0(\mathbf{x})$. In some sense, $\mathcal{G}(\mathbf{x})$ resembles a Green's function, which is widely used in studies on ordinary differential equations (ODEs) and partial differential equations (PDEs) \citep{duffy2015green}. Once $\mathcal{G}(\mathbf{x})$ is known, we can calculate the velocity field:   
%----------------------;
%  Equation: Velocity  ;
%----------------------;
\begin{align}
    \mathbf{v}(\mathbf{x})
    = - \frac{1}{\widetilde{\mu}_0(\mathbf{x})} 
    \,\mathbf{K}(\mathbf{x})\,\text{grad}\big[\mathcal{P}(\mathbf{x})\big]
    = - \Delta \mathcal{P} \, \frac{1}{\widetilde{\mu}_0(\mathbf{x})} \,\mathbf{K}(\mathbf{x}) \, \text{grad}\bigl[\mathcal{G}(\mathbf{x})\bigr]
\end{align}
In arriving at the above equation, we have used Eq.~\eqref{Eqn:Hopf_Cole_Delta_mathcalP} and the fact that $\mathcal{P}_{\mathrm{atm}}$ is a constant. Noting that $\Delta \mathcal{P}$ is independent of $\mathbf{x}$, the total volumetric flux at the production boundary $\Gamma_{\text{prod}}$ can be calculated as follows: 
%-----------------------------;
%  Equation: Volumetric flux  ;
%-----------------------------;
\begin{align}
    \label{Eqn:Hopf_Cole_Q_Linear}
    Q = \int_{\Gamma_{\text{prod}}}
    \mathbf{v}(\mathbf{x}) \bullet \widehat{\mathbf{n}}(\mathbf{x}) \; \mathrm{d}\Gamma
   = \mathcal{C} \, \Delta \mathcal{P} 
\end{align}
where 
%-----------------------------------;
%  Equation: Constant C definition  ;
%-----------------------------------;
\begin{align}
    \mathcal{C} := \int_{\Gamma_{\text{prod}}}
    - \frac{1}{\widetilde{\mu}_0(\mathbf{x})} \,
    \mathbf{K}(\mathbf{x}) \,
    \text{grad}\big[\mathcal{G}(\mathbf{x})\big] \bullet
    \widehat{\mathbf{n}}(\mathbf{x}) \; \mathrm{d}\Gamma
\end{align}
As is the case with $\mathcal{G}(\mathbf{x})$, $\mathcal{C}$ is also entirely determined by the geometry of the domain, permeability $\mathbf{K}(\mathbf{x})$, and $\widetilde{\mu}_0(\mathbf{x})$. 

Now, we execute Step 3: using Eqs.~\eqref{Eqn:Hopf_Cole_h_transformation} and \eqref{Eqn:Hopf_Cole_Q_Linear} we obtain a relationship between the volumetric flux at the production well $Q$ and the injection pressure $p_{\mathrm{inj}}$:
%------------------------;
%  Equation: Q vs p_inj  ;
%------------------------;
\begin{align}
    \label{Eqn:Hopf_Cole_Q_linear_Ceiling_Flux}
    Q = \frac{C\,p_{\mathrm{atm}}}{\beta} \left(1-\text{exp}\left[-\beta\left(\frac{p_{\mathrm{inj}}}{p_{\mathrm{atm}}} - 1\right)\right]\right)
\end{align}

The relationship satisfies the following properties: 
%----------------------------------------;
%  Enumerate: Properties of Q vs. p_inj  ;
%----------------------------------------;
\begin{enumerate}
    \item No volumetric flux at the production well when $p_{\mathrm{inj}} = p_{\mathrm{atm}}$.
    \item The volumetric flux increases monotonically with $p_{\mathrm{inj}}$. 
    \item The volumetric flux is bounded for large values of $p_{\mathrm{inj}}$. That is, 
    \begin{align}
        \lim_{p_{\mathrm{inj}}\rightarrow \infty} Q 
        = \frac{C \, p_{\mathrm{atm}}}{\beta} 
    \end{align}
    The quantity on the right side of the above equation is a bounded constant.
    \item The slope of the curve tends to zero as the injection pressure tends to infinity: 
    \begin{align}
        \lim_{p_{\mathrm{inj}}\rightarrow \infty} \frac{\mathrm{d} Q}{\mathrm{d} p_{\mathrm{inj}}}
        = \lim_{p_{\mathrm{inj}}\rightarrow \infty} 
        C \, \text{exp}\left[-\beta\left(\frac{p_{\mathrm{inj}}}{p_{\mathrm{atm}}} - 1\right)\right] = 0
    \end{align}
\end{enumerate}
This establishes the ceiling flux via an analytical means.

Figure \ref{Fig:Hopf_Cole_Ceiling_flux} verifies the prediction of the ceiling flux using the proposed approach. In addition, obtaining the data points to trace the $Q$ vs. $p_{\mathrm{inj}}$ curve under the proposed approach computationally marginal compared to the original approach, as described in the figure caption. 

%--------------------------;
%  Figure 7: Ceiling flux  ;
%--------------------------;
\begin{figure}[h!]
    \centering
    \includegraphics[width=0.6\linewidth]{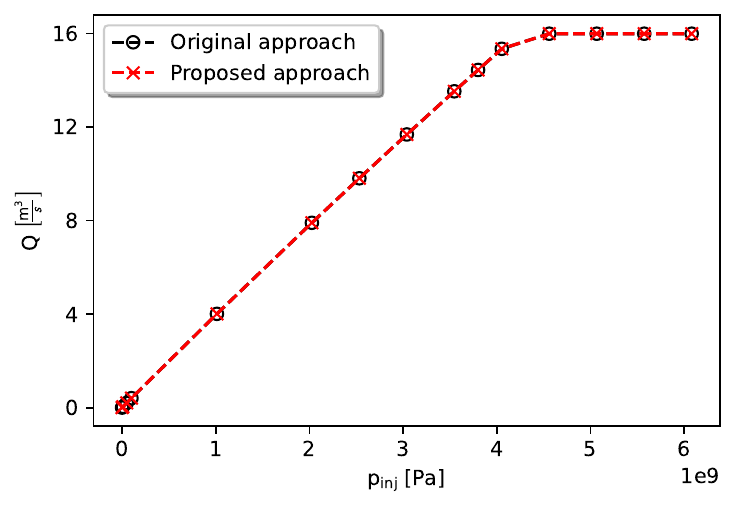}
    \caption{This figure demonstrates that the proposed approaches accurately capture the ceiling flux, exactly matching the results from the original method. However, the proposed approach offers a distinct advantage by requiring only one numerical simulation at an injection pressure other than $p_{\mathrm{atm}}$. Then, by applying Eq.~\eqref{Eqn:Hopf_Cole_Q_linear_Ceiling_Flux}, the ceiling flux at other injection pressures can be easily predicted without additional simulations, significantly reducing computational costs and avoiding the need for a nonlinear solver. In contrast, the original method requires separate numerical simulations for each injection pressure, making it expensive and reliant on a nonlinear solver.}
    \label{Fig:Hopf_Cole_Ceiling_flux}
\end{figure}

    %*********************************************;
%                                             ;
%  NAME                                       ;
%    S7_Hopf_Cole_Closure.tex                 ;
%                                             ;
%*********************************************;
\section{CLOSURE}
\label{Sec:S7_Hopf_Cole_Closure}

In this paper, we presented a framework based on the Hopf-Cole transformation for studying fluid flow through porous media, with a particular focus on pressure-dependent viscosity---a characteristic feature of most organic liquids. We considered the Barus model, which captures the experimentally observed exponential dependence of viscosity on fluid pressure. The proposed framework offers several significant advantages for modeling pressure-dependent flow through porous media, including:
%-------------------------------;
%  Enumerate: Salient features  ;
%-------------------------------;
\begin{enumerate}
    \item \textbf{Mathematical insights:} Important results, such as maximum principles, can be directly obtained for the Barus model from classical Darcy equation results, eliminating the need for advanced functional analysis tools such as quasilinear elliptic partial differential equation theory. This straightforward derivation holds immense pedagogical value in engineering instruction.
    \item \textbf{Numerical robustness:} The framework provides a reliable method for obtaining numerical solutions by leveraging well-established formulations for classical Darcy equations, thereby avoiding the convergence issues commonly encountered when solving nonlinear equations. This makes the framework a versatile and practical modeling tool for porous media applications involving pressure-sensitive fluid properties.
    \item \textbf{Identification of ill-posed problems:} We identify a class of velocity-driven problems that are ill-posed under the Barus model. The Hopf-Cole transformation clearly reveals the non-existence of solutions for this class of problems. Therefore, the Barus model should be used with care, especially in those situations where the flow through the porous medium is forced by controlling the velocity (i.e., flow rate) on a part of the boundary. 
    \item \textbf{Analytical derivation of ceiling flux:} Features such as ceiling flux under the Barus model (i.e., flux reaching a plateau as pressure increases) can be easily derived using the Hopf-Cole transformation approach, without requiring extensive numerical simulations, as was previously necessary.
    \item \textbf{Extension of mechanical principles:} Mechanical principles such as reciprocal relations, well established for classical Darcy equations (a system of linear partial differential equations), can now be extended to pressure-dependent models, which are inherently nonlinear. This marks the first instance of deriving reciprocal relations for a nonlinear problem, whereas all previous cases pertained to linear problems.
\end{enumerate}

    \appendix
    %*********************************************;
%                                             ;
%  NAME                                       ;
%    Hopf_Cole_Appendix.tex                   ;
%                                             ;
%*********************************************;
\section{Reciprocity for Darcy-type equations}
\label{App:Hopf_Cole_Reciprocal}
%==================================;
%  Theorem: A reciprocal relation  ;
%----------------------------------;
\begin{theorem}[A reciprocal relation on $\mathcal{P}(\mathbf{x})$]
    \label{Thm:Hopf_Cole_Betti_Theorem_Darcy}
    Let $\big(\mathbf{v}^{(1)}(\mathbf{x}),\mathcal{P}^{(1)}(\mathbf{x})\big)$ and $\big(\mathbf{v}^{(2)}(\mathbf{x}),\mathcal{P}^{(2)}(\mathbf{x})\big)$ be the solutions of Eqs.~\eqref{Eqn:Hopf_Cole_Linear_BoLM}--\eqref{Eqn:Hopf_Cole_Linear_BoM} corresponding to the prescribed boundary conditions $\big(v^{(1)}_n(\mathbf{x}),\mathcal{P}_{\mathrm{p}}^{(1)}(\mathbf{x})\big)$ and $\big(v^{(2)}_n(\mathbf{x}),\mathcal{P}_{\mathrm{p}}^{(2)}(\mathbf{x})\big)$, respectively. Then, these two sets of solutions and their corresponding prescribed boundary conditions satisfy the following relation:
    %-------------------------------------------;
    %  Equation: Reciprocal relation for Darcy  ;
    %-------------------------------------------;
    \begin{align}
        \label{Eqn:Hopf_Cole_Betti_Theorem_Darcy}
        \int_{\Gamma^{v}} v_{n}^{(2)}(\mathbf{x}) 
        \, \mathcal{P}^{(1)}(\mathbf{x}) \, \mathrm{d} \Gamma 
        - \int_{\Gamma^{p}}
        \mathcal{P}^{(2)}_{\mathrm{p}}(\mathbf{x}) 
        \, \mathbf{v}^{(1)}(\mathbf{x}) 
        \bullet \widehat{\mathbf{n}}(\mathbf{x}) 
        \, \mathrm{d} \Gamma 
        &=
        \int_{\Gamma^{v}}
        v^{(1)}_{n}(\mathbf{x}) \, \mathcal{P}^{(2)}(\mathbf{x}) 
        \, \mathrm{d} \Gamma \nonumber \\ 
        &\hspace{0.4in} - \int_{\Gamma^{p}}
        \mathcal{P}^{(1)}_{\mathrm{p}}(\mathbf{x}) 
        \, \mathbf{v}^{(2)}(\mathbf{x}) 
        \bullet \widehat{\mathbf{n}}(\mathbf{x}) 
        \, \mathrm{d} \Gamma 
    \end{align}
\end{theorem}

%--------------------------------------------;
%  Remark: Remark about boundary partitions  ;
%--------------------------------------------;
\begin{remark}
    Both solutions $\big(\mathbf{v}^{(1)}(\mathbf{x}),\mathcal{P}^{(1)}(\mathbf{x})\big)$ and $\big(\mathbf{v}^{(2)}(\mathbf{x}),\mathcal{P}^{(2)}(\mathbf{x})\big)$ assume the same boundary partitions to prescribe their respective boundary conditions. 
\end{remark}

%======================================;
%  Proof of linear reciprocal theorem  ;
%--------------------------------------;
\begin{proof}
    Let $\big(\mathcal{P}^{(1)}(\mathbf{x}), \mathbf{v}^{(1)}(\mathbf{x})\big)$ and $\big(\mathcal{P}^{(2)}(\mathbf{x}), \mathbf{v}^{(2)}(\mathbf{x})\big)$ be the solutions corresponding to the two sets of prescribed boundary conditions: $\big(\mathcal{P}_{\mathrm{p}}^{(1)}(\mathbf{x}), v_{n}^{(1)}(\mathbf{x})\big)$ and $\big(\mathcal{P}_{\mathrm{p}}^{(2)}(\mathbf{x}), v_{n}^{(2)}(\mathbf{x})\big)$. We assume the boundary partitions---$\Gamma^{p}$ and $\Gamma^{v}$---to the same for both solution fields. 

    These two sets of solution fields satisfy the following governing equations: 
    %------------------------------------;
    %  Equation: BVPs for two solutions  ;
    %------------------------------------;
    \begin{subequations}
        %--------------------;
        %  Equation: BVP #1  ;
        %--------------------;
        \begin{alignat}{2}
            \label{Eqn:Hopf_Cole_Linear_Betti_BoLM_1}
            &\widetilde{\mu}_{0}(\mathbf{x}) \, 
            \mathbf{K}^{-1}(\mathbf{x}) \, \mathbf{v}^{(1)}(\mathbf{x})
            + \text{grad}\big[\mathcal{P}^{(1)}(\mathbf{x})\big] 
            = \mathbf{0} 
            &&\qquad \mathrm{in} \; \Omega \\
            %%%
            \label{Eqn:Hopf_Cole_Linear_Betti_BoM_1}
            &\text{div}\big[\mathbf{v}^{(1)}(\mathbf{x})\big] = 0 
            &&\qquad \mathrm{in} \; \Omega \\
            %%%
            \label{Eqn:Hopf_Cole_Linear_Betti_pBC_1}
            &\mathcal{P}^{(1)}(\mathbf{x}) = \mathcal{P}^{(1)}_{\mathrm{p}}(\mathbf{x}) 
            &&\qquad \mathrm{on} \; \Gamma^{p} \\
            %%%
            \label{Eqn:Hopf_Cole_Linear_Betti_vBC_1}
            &\mathbf{v}^{(1)}(\mathbf{x}) \bullet \widehat{\mathbf{n}}(\mathbf{x}) = v_n^{(1)}(\mathbf{x})
            &&\qquad \mathrm{on} \; \Gamma^{v} 
            %%%
        \end{alignat}
        %--------------------;
        %  Equation: BVP #2  ;
        %--------------------;
        \begin{alignat}{2}
            \label{Eqn:Hopf_Cole_Linear_Betti_BoLM_2}
            &\widetilde{\mu}_{0}(\mathbf{x}) \, 
            \mathbf{K}^{-1}(\mathbf{x}) \, \mathbf{v}^{(2)}(\mathbf{x})
            + \text{grad}\big[\mathcal{P}^{(2)}(\mathbf{x})\big] 
            = \mathbf{0} 
            &&\qquad \mathrm{in} \; \Omega \\
            %%%
            \label{Eqn:Hopf_Cole_Linear_Betti_BoM_2}
            &\text{div}\big[\mathbf{v}^{(2)}(\mathbf{x})\big] = 0 
            &&\qquad \mathrm{in} \; \Omega \\
            %%%
            \label{Eqn:Hopf_Cole_Linear_Betti_pBC_2}
            &\mathcal{P}^{(2)}(\mathbf{x}) = \mathcal{P}^{(2)}_{\mathrm{p}}(\mathbf{x}) 
            &&\qquad \mathrm{on} \; \Gamma^{p} \\
            %%%
            \label{Eqn:Hopf_Cole_Linear_Betti_vBC_2}
            &\mathbf{v}^{(2)}(\mathbf{x}) \bullet \widehat{\mathbf{n}}(\mathbf{x}) = v_n^{(2)}(\mathbf{x})
            &&\qquad \mathrm{on} \; \Gamma^{v} 
            %%%
        \end{alignat}
        %%%
    \end{subequations}

    Taking a dot product between $\mathbf{v}^{(2)}(\mathbf{x})$ and Eq.~\eqref{Eqn:Hopf_Cole_Linear_Betti_BoLM_1}, we write the following:
    %---------------------;
    %  Equation: Step 1a  ;
    %---------------------;
    \begin{equation}
        \label{Eqn:Hopf_Cole_Linear_Betti_proof_Step_1a}
        \mathbf{v}^{(2)}(\mathbf{x}) \bullet \Big(\widetilde{\mu}_{0}(\mathbf{x}) \, \mathbf{K}^{-1}(\mathbf{x}) \, \mathbf{v}^{(1)} (\mathbf{x}) + \text{grad}\big[\mathcal{P}^{(1)}(\mathbf{x})\big]\Big) = 0 
    \end{equation}
    Likewise, taking a dot product between $\mathbf{v}^{(1)}(\mathbf{x})$ and Eq.~\eqref{Eqn:Hopf_Cole_Linear_Betti_BoLM_2}, we have: 
    %---------------------;
    %  Equation: Step 1b  ;
    %---------------------;
    \begin{equation}
        \label{Eqn:Hopf_Cole_Linear_Betti_proof_Step_1b}
        \mathbf{v}^{(1)}(\mathbf{x}) \bullet \Big(\widetilde{\mu}_{0}(\mathbf{x}) \, \mathbf{K}^{-1}(\mathbf{x}) \, \mathbf{v}^{(2)} (\mathbf{x}) + \text{grad}\big[\mathcal{P}^{(2)}(\mathbf{x})\big]\Big) = 0
    \end{equation}
    Subtracting the two preceding two equations and noting the symmetry of $\mathbf{K}(\mathbf{x})$ (i.e., the second-order permeability tensor), we arrive at the following: 
    %--------------------;
    %  Equation: Step 2  ;
    %--------------------;
    \begin{equation}
        \label{Eqn:Hopf_Cole_Linear_Betti_proof_Step_2}
        \mathbf{v}^{(1)}(\mathbf{x}) \bullet \text{grad}\big[\mathcal{P}^{(2)}(\mathbf{x})\big] = \mathbf{v}^{(2)}(\mathbf{x}) \bullet \text{grad}\big[\mathcal{P}^{(1)}(\mathbf{x})\big]
    \end{equation}
    
    Integrating over the domain and applying Green’s identity, we rewrite the above equation in the following equivalent form:
    %--------------------;
    %  Equation: Step 3  ;
    %--------------------;
    \begin{align}
        \label{Eqn:Hopf_Cole_Linear_Betti_proof_Step_3}
        &\int_{\partial \Omega} \mathbf{v}^{(1)}(\mathbf{x}) 
        \bullet \widehat{\mathbf{n}}(\mathbf{x}) \, 
        \mathcal{P}^{(2)}(\mathbf{x}) \, \mathrm{d} \Gamma 
       - \int_{\Omega} \mathrm{div}\big[\mathbf{v}^{(1)}(\mathbf{x})\big] \, \mathcal{P}^{(2)}(\mathbf{x}) 
        \, \mathrm{d} \Omega \nonumber \\ 
        %%%
        &\qquad \qquad = \int_{\partial \Omega} \mathbf{v}^{(2)}(\mathbf{x}) 
        \bullet \widehat{\mathbf{n}}(\mathbf{x}) \, 
        \mathcal{P}^{(1)}(\mathbf{x}) \, \mathrm{d} \Gamma 
        - \int_{\Omega} \mathrm{div}\big[\mathbf{v}^{(2)}(\mathbf{x})\big] \, \mathcal{P}^{(1)}(\mathbf{x}) 
        \, \mathrm{d} \Omega 
    \end{align}
    Now, invoking the incompressibility constraints, given by Eqs.~\eqref{Eqn:Hopf_Cole_Linear_Betti_BoM_1} and \eqref{Eqn:Hopf_Cole_Linear_Betti_BoM_2}, we get:
    %--------------------;
    %  Equation: Step 4  ;
    %--------------------;
    \begin{align}
        \int_{\partial \Omega} \mathbf{v}^{(1)}(\mathbf{x}) 
        \bullet \widehat{\mathbf{n}}(\mathbf{x}) \, 
        \mathcal{P}^{(2)}(\mathbf{x}) \, \mathrm{d} \Gamma 
        = \int_{\partial \Omega} \mathbf{v}^{(2)}(\mathbf{x}) 
        \bullet \widehat{\mathbf{n}}(\mathbf{x}) \, 
        \mathcal{P}^{(1)}(\mathbf{x}) \, \mathrm{d} \Gamma 
    \end{align}
    Splitting the integrals over the boundary partitions and enforcing the associated boundary conditions (i.e., Eqs.~\eqref{Eqn:Hopf_Cole_Linear_Betti_pBC_1}, \eqref{Eqn:Hopf_Cole_Linear_Betti_vBC_1}, \eqref{Eqn:Hopf_Cole_Linear_Betti_pBC_2}, and \eqref{Eqn:Hopf_Cole_Linear_Betti_vBC_2}), we obtain:
    %---------------------------------------------------;
    %  Equation: Final form of the reciprocal relation  ;
    %---------------------------------------------------;
    \begin{align}
        \label{Eqn:Hopf_Cole_Linear_Betti_Theorem}
        &\int_{\Gamma^{v}} v_{n}^{(1)}(\mathbf{x}) 
        \, \mathcal{P}^{(2)}(\mathbf{x}) \, \mathrm{d} \Gamma 
        + \int_{\Gamma^{p}}
        \, \mathbf{v}^{(1)}(\mathbf{x}) 
        \bullet \widehat{\mathbf{n}}(\mathbf{x}) \, 
        \mathcal{P}^{(2)}_{\mathrm{p}}(\mathbf{x}) 
        \, \mathrm{d} \Gamma \nonumber \\ 
        %%%
        &\qquad \qquad = \int_{\Gamma^{v}}
        v^{(2)}_{n}(\mathbf{x}) \, \mathcal{P}^{(1)}(\mathbf{x}) 
        \, \mathrm{d} \Gamma
        + \int_{\Gamma^{p}}
        \mathbf{v}^{(2)}(\mathbf{x}) 
        \bullet \widehat{\mathbf{n}}(\mathbf{x}) \, 
        \mathcal{P}^{(1)}_{\mathrm{p}}(\mathbf{x}) 
        \, \mathrm{d} \Gamma 
    \end{align}
    Rearranging the terms yields the desired reciprocal relation given by Eq.~\eqref{Eqn:Hopf_Cole_Betti_Theorem_Darcy}.
    %%%
\end{proof}

    %*********************************************;
%                                             ;
%  NAME                                       ;
%    Hopf_Cole_Appendix_B.tex                 ;
%                                             ;
%*********************************************;
\section{Kirchhoff transformation}
\label{App:Hopf_Cole_Kirchhoff}

The Kirchhoff transformation introduces a new variable based on solution-dependent material properties \citep{carslaw1984conduction}. Since in our case, only viscosity depends on the pressure, we define a new field variable as follows: 
%----------------------------------;
%  Equation: Transformed variable  ;
%----------------------------------;
\begin{align}
    \mathcal{P}_{\mathrm{K}}(\mathbf{x}) 
    := \int_{p_0}^{\widetilde{p}(\mathbf{x})} \frac{\mu(p_0)}{\mu(p^\prime)} \, \mathrm{d}p^\prime
    = \int_{p_0}^{\widetilde{p}(\mathbf{x})} 
    \frac{1}{g(p^\prime)} \, \mathrm{d}p^\prime
\end{align}
where the function $g(\cdot)$ is defined in Eq.~\eqref{Eqn:Hopf_Cole_g_of_ptilde}, and $\mu_0$ is the viscosity at pressure $p_0$. Executing the integral, we get the expression for the transformed variable under the Kirchhoff transformation:
%------------------------;
%  Equation: mathcal_PK  ;
%------------------------;
\begin{align}
    \label{Eqn:Hopf_Cole_Kirchhoff_expression}
    \mathcal{P}_{\mathrm{K}}(\mathbf{x})
    = \int_{p_0}^{\widetilde{p}(\mathbf{x})} 
    \exp\left[-\beta \, \left(\frac{p^{'}}{p_0} - 1\right)\right] \, 
    \mathrm{d} p^{'}
    = \frac{p_0}{\beta}
    \left(1 - 
    \exp\left[-\beta \, \left(\frac{\widetilde{p}(\mathbf{x})}{p_0} 
    - 1\right)\right] \right) 
\end{align}
For completeness, we also record the inverse (i.e., going from $\mathcal{P}_{\mathrm{K}}(\mathbf{x})$ to $\widetilde{p}(\mathbf{x})$):
%-------------------------------------------;
%  Equation: p_tilde in terms of mathcal_P  ;
%-------------------------------------------;
\begin{align}
    \widetilde{p}(\mathbf{x}) 
    = p_0 \left(1 - \frac{1}{\beta} \, \ln\left[1 - \frac{\beta}{p_0} \, \mathcal{P}_{\mathrm{K}}(\mathbf{x})\right] \right)
\end{align}
The expression derived from the Hopf-Cole transformation, used in the main part of this paper, is notably simpler than the one presented above (cf. Eqs.~\eqref{Eqn:Hopf_Cole_Functional_form_of_mathcal_P} and \eqref{Eqn:Hopf_Cole_Kirchhoff_expression}).

We now derive the transformed governing equations based on the Kirchhoff transformation. Using the Leibniz integral rule, which provides a formula for differentiation under an integral sign \citep{kaplan1973advanced}, we proceed as follows:  
%--------------------------;
%  Equation: Leibniz rule  ;
%--------------------------;
\begin{align}
    \mathrm{grad}\big[\mathcal{P}_{\mathrm{K}}(\mathbf{x})\big] 
    &= \frac{1}{g\big(\widetilde{p}(\mathbf{x})\big)} \, 
    \mathrm{grad}\big[\widetilde{p}(\mathbf{x})\big] \\
    &= \frac{\widetilde{\mu}_0(\mathbf{x})}{\widetilde{\mu}_0(\mathbf{x}) \, g\big(\widetilde{p}(\mathbf{x})\big)} \, 
    \mathrm{grad}\big[\widetilde{p}(\mathbf{x})\big] \\
    &= \frac{\widetilde{\mu}_0(\mathbf{x})}{\widehat{\mu}\big(p(\mathbf{x})\big)} \, 
    \mathrm{grad}\big[\widetilde{p}(\mathbf{x})\big] 
\end{align}
where the function form of $\widehat{\mu}\big(p(\mathbf{x})\big)$ is provided in Eq.~\eqref{Eqn:Hopf_Cole_Barus_model}. 
Applying the balance of linear momentum given by Eq.~\eqref{Eqn:Hopf_Cole_BoLM_Conservative}, we write: 
%---------------------------;
%  Equation: Applying BoLM  ;
%---------------------------;
\begin{align}
    \mathrm{grad}\big[\mathcal{P}_{\mathrm{K}}(\mathbf{x})\big]
    = \widetilde{\mu}_{0}(\mathbf{x}) \, \mathbf{K}^{-1}(\mathbf{x}) 
    \, \left(\frac{1}{\widehat{\mu}\big(p(\mathbf{x})\big)} \, \mathbf{K}(\mathbf{x}) 
    \, \mathrm{grad}\big[\widetilde{p}(\mathbf{x})\big] \right)
    = -\widetilde{\mu}_{0}(\mathbf{x}) \, \mathbf{K}^{-1}(\mathbf{x}) 
    \, \mathbf{v}(\mathbf{x})
\end{align}

Then, the governing equations in terms of $\mathcal{P}_{\mathrm{K}}(\mathbf{x})$ and $\mathbf{v}(\mathbf{x})$ take the following form:
%-----------------------------------;
%  Equation: Transformed equations  ;
%-----------------------------------;
\begin{subequations}
    \begin{alignat}{2}
        &\widetilde{\mu}_0(\mathbf{x}) \, \mathbf{K}^{-1}(\mathbf{x}) \, 
        \mathbf{v}(\mathbf{x}) 
        + \mathrm{grad}\big[\mathcal{P}_{\mathrm{K}}(\mathbf{x})\big] = \mathbf{0} 
        &&\quad \mathrm{in} \; \Omega \\ 
        &\mathrm{div}\big[\mathbf{v}(\mathbf{x})\big] = 0 
        &&\quad \mathrm{in} \; \Omega \\ 
        &\mathcal{P}_{\mathrm{K}}(\mathbf{x}) = \mathcal{P}_{\mathrm{K}\mathrm{p}}(\mathbf{x}) 
        &&\quad \mathrm{on} \; \Gamma^{p} \\ 
        & \mathbf{v}(\mathbf{x}) \bullet \widehat{\mathbf{n}}(\mathbf{x}) = v_n(\mathbf{x}) 
        &&\quad \mathrm{on} \; \Gamma^{v} 
    \end{alignat}
\end{subequations}
The above boundary value problem again resembles the classical Darcy equations. Note the definition of $\mathcal{P}_{\mathrm{K}\mathrm{p}}(\mathbf{x})$, which is different from that of $\mathcal{P}_{\mathrm{p}}(\mathbf{x})$ (given by Eq.~\eqref{Eqn:Hopf_Cole_mathcalP_p_definition}). Since the transformations are different (cf. Eqs.~\eqref{Eqn:Hopf_Cole_h_transformation} and \eqref{Eqn:Hopf_Cole_Kirchhoff_expression}), $\mathcal{P}_{\mathrm{p}}(\mathbf{x})$ (under the Hopf-Cole transformation) and $\mathcal{P}_{\mathrm{K}\mathrm{p}}(\mathbf{x})$ (under the Kirchhoff transformation) differ for the same prescribed pressure boundary condition in the original problem (i.e., $p_{\mathrm{p}}(\mathbf{x}))$. 

We now demonstrate that the Kirchhoff transformation can be derived as a special case of the Hopf-Cole transformation, a connection first noted by \citet{Vadasz2010} in the context of heat conduction. To illustrate this, consider the general expression under the Hopf-Cole transformation given by Eq.~\eqref{Eqn:Hopf_Cole_f_in_terms_of_A_and_B}, which contains two integration constants that can be freely chosen. By selecting the following specific values in Eq.~\eqref{Eqn:Hopf_Cole_f_in_terms_of_A_and_B}:
%--------------------------------------------------;
%  Equation: A and B for Kirchhoff transformation  ;
%--------------------------------------------------;
\begin{align}
    A= 1 \quad \mathrm{and} \quad B = -\frac{p_0}{\beta}
\end{align} 
we recover the expression for the Kirchhoff transformation (i.e., Eq.~\ref{Eqn:Hopf_Cole_Kirchhoff_expression}).
As mentioned earlier, we selected $A = 1$ and $B = 0$ in order to derive the transformation utilized in the main part of the paper; see Eq.~\eqref{Eqn:Hopf_Cole_selection_of_A_and_B}.

Notably, alternative choices for these constants lead to a hierarchy of possible transformations. This highlights the greater flexibility of the Hopf-Cole approach, which allows for a simpler and more versatile transformation than the Kirchhoff approach, simply by appropriately selecting the values of the constants.

    %=====================;
    %  Data availability  ;
    %=====================;
    \section*{DATA AVAILABILITY}
    Data supporting the findings of this study are available from the corresponding author on request.

    %================;
    %  Bibliography  ;
    %================;
    \bibliographystyle{plainnat}
    \bibliography{Master_References}
    %%%
\end{document}